\documentclass[12pt]{amsart}
\usepackage[utf8]{inputenc}
\usepackage{amsmath,amsthm,amssymb}
\usepackage{xcolor}
\usepackage[shortlabels]{enumitem}
\usepackage{hyperref}
\usepackage{mathtools}
\usepackage{graphicx}
\newtheorem{theorem}{Theorem}[section]
\newtheorem{conj}[theorem]{Conjecture}
\newtheorem{defn}[theorem]{Definition}
\newtheorem{definition}[theorem]{Definition}

\newtheorem{remark}[theorem]{Remark}
\newtheorem{lemma}[theorem]{Lemma}
\newtheorem{proposition}[theorem]{Proposition}
\newtheorem{prop}[theorem]{Proposition}
\newtheorem{cor}[theorem]{Corollary}
\newtheorem{corollary}[theorem]{Corollary}
\mathtoolsset{showonlyrefs}

\newcommand{\R}{\mathbb{R}}\newcommand{\C}{\mathbb{C}}
\newcommand{\T}{\mathbb{T}}\newcommand{\D}{\mathbb{D}}

\newcommand{\E}{\mathbb{E}}

\newcommand{\les}{\lessapprox}

\newcommand{\ges}{\gtrapprox}

\newcommand{\cS}{\mathcal{S}}

\newcommand{\cP}{\mathcal{P}}
\newcommand{\cT}{\mathcal{T}}
\newcommand{\cA}{\mathcal{A}}
\newcommand{\cD}{\mathcal{D}}
\newcommand{\cL}{\mathcal{L}}
\newcommand{\cN}{\mathcal{N}}
\newcommand{\N}{\mathbb{N}}

\newcommand{\cF}{\mathcal{F}}
\newcommand{\cU}{\mathcal{U}}
\newcommand{\cG}{\mathcal{G}}
\newcommand{\cB}{\mathcal{B}}

\newcommand{\PP}{\mathbb{P}}
\newcommand{\cH}{\mathcal{H}}

\newcommand{\cI}{\mathcal{I}}

\newcommand{\cQ}{\mathcal{Q}}

\newcommand{\dist}{\mathrm{dist}}

\newcommand{\SSS}{\mathbb{S}}

\newcommand{\dem}{\de^{-O(\e)}}
\newcommand{\dep}{\de^{O(\e)}}

\newcommand{\cR}{\mathcal{R}}

\newcommand{\cV}{\mathcal{V}}

\newcommand{\lbox}{\underline\dim_{\mathrm{B}}}
\newcommand{\ubox}{\overline\dim_{\mathrm{B}}}

\newcommand{\cont}[1]{\mathcal{H}^{#1}_{[\delta,\infty)}}

\newcommand{\de}{\delta}
\newcommand{\e}{\epsilon}

\newcommand{\diam}{\operatorname{diam}}

\newcommand{\dimh}{\dim_\mathrm{H}}

\newcommand{\cns}[2]{N_{{#1}}({#2})}

\newcommand{\cov}[1]{N_\de({#1})}
\newcommand{\cn}[1]{N_\de({#1})}
\newtheorem*{ack*}{Acknowledgment}

\author{Ciprian Demeter}\address{Department of Mathematics, Indiana University, 831 East 3rd St., Bloomington IN 47405}
\email{demeterc@iu.edu}
\author{William O'Regan}\address{Department of Mathematics, The University of British Columbia, Room 121, 1984 Mathematics Road, Vancouver  BC Canada V6T 1Z2}
\email{woregan@math.ubc.ca}

\begin{document}
	\title{Incidence estimates for quasi-product sets and applications}

	\thanks{CD is partially supported by the NSF grant  DMS-2349828, WOR is supported in part by an NSERC Alliance grant administered by Pablo Shmerkin and Joshua Zahl}
	\begin{abstract}
		We use recent advances in the theory of Furstenberg sets to prove new incidence results of Szemer\'edi--Trotter strength for $\delta$-discretized structures with Cartesian product flavor. We use these results to make progress on a number of problems that include energy estimates and Fourier decay of fractal measures supported on curves,  as well as various sum-product-like results governed by fractal dimension.  
	\end{abstract}

	\maketitle
	
	\section{Introduction}
	
	\subsection{Notation} We introduce the collection of  $\delta$-squares in $[0,1]^d$  $$\cD_\delta=\{[n_1\delta,(n_1+1)\delta]\times \cdots \times [n_d\delta,(n_d+1)\delta]:\;0\le n_1,\ldots n_d\le \delta^{-1}-1\}.$$

	We use $|\;\;|$ to denote Lebesgue measure and $\#$ for the counting measure. For $\de > 0$ we use $\cn{S}$ to denote the $\de$-covering number of $S$, and $\cN_\de(S)$ for the $\de$-neighborhood  of $S$. 
	
	If quantities $A,B$ depend on $\delta$, we write $A \lesssim B$ to mean there exists $C > 0$ so that $A(\delta) \leq CB(\delta)$ for all $\delta > 0.$  We use $\lesssim_\e C$ if we wish to emphasize that the suppressed constant depends on the parameter $\e.$ We write $A\sim B$ if $A\lesssim B$ and $B\lesssim A$. 
	
	We write $A\les B$ to mean that $A\lesssim_\upsilon \delta^{-\upsilon} B$ for each $\upsilon>0$. Also, $A\approx B$ means that both $A\les B$ and $B\les A$ hold.
	
	The function $\log$ will be to base 2. For a set $A \subset \R^d$ and $\delta >0$ we let  $A_\de$ denote a $\de$-separated subset of $A$ with $\#A_\de \sim \cn A.$ Write the Hausdorff content of $A$ at dimension $0 \leq s \leq d$ at scale $\delta > 0$ by
	\begin{equation}
		\cH_{[\delta,\infty)}^s(A) := \inf\bigg\{\sum_{U \in \cU} \diam(U)^s : A \subset \bigcup_{U \in \cU}U \text{ and } \diam(U) \in \{\de,2\de,\ldots\} \text{ for all } U \in \cU\bigg\}.
	\end{equation}
	All sets $A$ that we consider will have diameter $O(1)$, so the sets $U$ in each cover will have diameter $\diam(U)\lesssim 1$.
	
	\begin{definition}
		Let $0 \leq s \leq d, C, \delta >0.$ We say that $A \subset \R^d$ is a $(\delta,s,C)$-\textit{set} if 
		\begin{equation}
			\cn{A \cap B(x,r)} \leq Cr^s\cn{A}\text{ for all } x \in \R^d,\; r \geq \delta.
		\end{equation}
		In some literature these sets are called Frostman sets.
		
		We say that $A \subset \R^d$ is a $(\delta,s,C)$-\textit{KT set} if 
		\begin{equation}
			\cn{A \cap B(x,r)} \leq C(r/\delta)^s \text{ for all } x \in \R^d,\; r \geq \delta.
		\end{equation}
		In some literature these sets are called Katz-Tao sets.

		If $C \sim 1$ then we will often drop the $C$ from the notation and simply refer to $(\delta,s)$-sets and $(\de,s)$-KT sets.
	\end{definition}
	\subsection{Incidence estimates}
	\begin{defn}\label{KK}
		Given $K_1,K_2\ge 1$ and $s,d\in[0,1]$ we call a  collection $\T$ of $\delta$-tubes a $(\delta,s,d,K_1,K_2)$-quasi-product set if its direction set $\Lambda$ is a $(\delta,s,K_1)$-KT set and for each $\theta\in\Lambda$, $\cT_\theta$ is a $(\delta,d,K_2)$-KT set.
		
		Likewise, we call a collection $\cP \subset \R^2$ a $(\de,s,d,K_1,K_2)$-quasi-product set if the projection of $\cP$ to the abscissa is a $(\de,s,K_1)$-KT set, and each vertical fiber is a $(\de,d,K_2)$-set.
	\end{defn}
	A $\delta$-square $p$ is said to be incident to a $\delta$-tube $T$ if $p\cap T\not=\emptyset$.
	A shading  of $\T$ is a map $Y$ on $\T$ such that for each $T\in\T$,  $Y(T)$ is a union of pairwise disjoint  $\delta$-squares in $\cD_\delta$ incident to $T$.
	We let
	$$Y(\T)=\cup_{T\in\T}Y(T),$$
	and write $\#Y(T)$, $\#Y(\T)$ for the total number of (distinct) $\delta$-squares in the collections $Y(T)$ and $Y(\T)$, respectively.
	We caution that $p\in Y(\T)$ and $p\cap T\not=\emptyset$ does not necessarily imply that $p\in Y(T)$.
	
	For $0<\epsilon_2 < \epsilon_1\ll 1$ we say that a shading is $(\epsilon_1,\epsilon_2)$-\textit{two-ends} if 
	\begin{equation}
		\#(Y(T) \cap B(x,\delta^{\epsilon_1})) \leq \delta^{\epsilon_2}\#Y(T) \text{ for all } x \in \R^2, \;T \in \T.
	\end{equation}
	
	For $0<s,d<1$ we let $$\sigma=\min\{s+d,2-s-d\}.$$
	Here is the first incidence result we are proving.
	\begin{theorem}
		\label{maininct}\label{thm.product}
		Let $0<s,d<1$.
		Define
		$$\alpha=\min\{3,2+\frac{d}{s}\}.$$
		Assume $\T$ is a $(\delta,s,d,K_1,K_2)$-quasi-product set. Consider a shading $Y$ of $\T$ such that $Y(T)$ is a $(\delta,\sigma,K_3)$-KT set for each $T\in\T$.
		
		Then  we have
		$$\sum_{T\in\T}\#Y(T)\les K_3^{\frac13}(K_1K_2)^{1-\frac1\alpha} (\delta^{-s-d}\#\T)^{1/{\alpha}}\# Y(\T)^{1-\frac1\alpha}.$$	
	\end{theorem} 
	The case $s+d=1$, $s\le d$ was solved in \cite{demwangszem}. Since $\sigma=1$ in that case, the requirement about $Y(T)$ being a $(\delta,\sigma,K_3)$-KT set is automatically satisfied with any $K_3\ge 1$.
	
	To better digest the strength of our estimate when $K_1=K_2=K_3\sim 1$, let us compare it with the Furstenberg set estimate in \cite{renwang}. Let us assume each $Y(T)$ is the union of $\sim M$ many $\delta$-squares. Let us also call $P$ the number of $\delta$-squares in $Y(\T)$. 
	
	Under the additional assumption that $\#\T\sim \delta^{-s-d}$ our estimate is equivalent to 
	\begin{equation}
		\label{ejifufu8gu8t}    
		P\ges M^{\frac\alpha{\alpha-1}}\#\T^{\frac{\alpha-2}{\alpha-1}}.
	\end{equation}
	If the additional assumption 
	$M\sim \delta^{-\sigma}$ is also added, but the product structure of tubes is no longer enforced,  then the main result in \cite{renwang} implies that
	\begin{equation}
		\label{ejifufu8gu8t2} 
		P\ges M^{\frac32}\#\T^{\frac12}.
	\end{equation}
	The two estimates can be seen to be the same if either $\alpha=3$ or $d+s\le 1$. When 
	$\alpha<3$ and  $d+s> 1$, \eqref{ejifufu8gu8t2} is stronger than \eqref{ejifufu8gu8t}.  
	
	The value of our estimate in Theorem \ref{maininct} lies in the fact that we do not assume that $M$ (or $\#\T$, but this is less relevant, due to Proposition \ref{lkogjthuyuh09yu}) is maximal. This is the distinction between $(\delta,\sigma)$-KT sets and $(\delta,\sigma)$-sets. 	
	\medskip
	
	Theorem \ref{maininct} should be also be compared with the sharp incidence bounds obtained in \cite{furen}. For example, there they obtained that for $0 < t < 1,$ if $\cP$ is a $(\delta,t)$-KT set of $\delta$-balls and $\cT$ is a $(\delta,t)$-KT set of $\delta$-tubes, then the number of incidences satisfies
	\begin{equation}\label{eq.furen}
		I(\cP,\cT) \lessapprox \delta^{-3t/2}.
	\end{equation}
	However, if the tubes form a $(\delta,t/2,t/2)$-quasi-product set,  $Y(T)$ is a $(\delta,\sigma)$-KT set and $\cP=Y(\cT)$ is only assumed to have size $\lesssim \delta^{-t}$ then Theorem \ref{maininct} gives the stronger bound
	\begin{equation}
		\sum_{T\in\cT}\#Y(T) \lessapprox \delta^{-4t/3}.
	\end{equation}
	In \cite{furen}, the sharp examples to \eqref{eq.furen} are when the balls and tubes form `cliques', where there are many balls inside a rectangular strip, and many tubes passing through this strip - see \cite[Construction 1]{furen}. The work of \cite{orpyi} shows that if $I(\cP,\cT) \approx \delta^{-3t/2}$ then $\cP$ and $\cT$ essentially forms such a configuration.

	\medskip
	
	We now introduce a type of structure that is critical  for the second incidence theorem we are proving. It has a different type of product flavor.
	\begin{definition}\label{def.rect}
		We call a set $S\subset [0,1]^2$ a rectangular $(\delta,u, C)$-KT set if for each $\delta\le r',r$ and each rectangle $B$ with dimensions $(r',r)$ and arbitrary orientation we have
		$$\#(S\cap B)\le C(\frac{\sqrt{rr'}}{\delta})^{u}.$$
	\end{definition}	
	Here is our second result. 
	\begin{theorem}
		\label{thm.rectangles1}
		Assume $0\le s\le 1$.	
		Consider a collection $\T$ of $\delta$-tubes and a collection $\PP$ of $\delta$-squares.
		
		Assume that  $\PP$ is a rectangular $(\delta,2s)$-KT set for each $p\in\PP$. Assume that $Y'(p)=\{T\in\T:\;p\cap T\not=\emptyset\}$ is a  $(\delta,\sigma)$-KT set, $\sigma=\min(2s,2-2s)$.
		Write
		$$I(\T,\PP)=\sum_{p\in\PP}\#Y'(p).$$
		Then 
		\begin{equation}
			I(\T,\PP)\les \delta^{-\frac{2s}{3}}(\#\T)^{2/3}(\#\PP)^{1/3}.
		\end{equation}
	\end{theorem}
	Theorem \ref{thm.rectangles1} will be both proved and used in Section \ref{sec:4} to estimate the Fourier decay of fractal measures supported on the parabola. 
	\medskip

	At the heart of our arguments for both Theorem \ref{maininct} and \ref{thm.rectangles1} lies the following result in \cite{wangwufurst}. The proof of Theorem \ref{tWW25} partially dwells on some of the key ingredients that were developed in \cite{demwangszem} for the purpose of proving the main result therein (the case $s+d=1$, $s\le d$ of Theorem \ref{maininct}). However, Theorem \ref{tWW25} covers more ground, and in particular will be shown to imply the one in  \cite{demwangszem}. See Section  \ref{sect.inc} for the details of this argument.
	\begin{theorem}
		\label{tWW25}
		Let $\T$ be a $(\delta,t,1)$-KT set for some $t\in(0,2)$, and let $0<\epsilon_2<\epsilon_1\ll 1$. Write $\sigma=\min(t,2-t)$.
		Let $Y$ be a shading of $\T$.
		Assume that each $Y(T)$ is a $(\delta,\sigma,K_3)$-KT set and also $(\epsilon_1,\epsilon_2)$-two-ends.
		
		Then for each $\epsilon>0$
		$$\sum_{T\in\T}\#Y(T)\le C(\epsilon, \epsilon_1,\epsilon_2)K_3^{1/3}\delta^{-O(\epsilon_1)-\epsilon}(\delta^{-t}\#Y(\T))^{2/3}.$$
	\end{theorem}
	\begin{proof}
		We include a short argument to show how this follows from Theorem 1.5 in \cite{wangwufurst}. It suffices to prove the inequality under the extra assumption that $\#Y(T)\sim \lambda\delta^{-1}$, for some $\lambda\lesssim 1$.  
		
		Theorem 1.5 in \cite{wangwufurst} proves that
		\begin{equation}
			\label{jefhureyg7yg78yg78y}
			(\frac\lambda\delta)^{1/2}\delta^{t/2}\sum_{T\in\T}\#Y(T)\le C(\epsilon, \epsilon_1,\epsilon_2)K_3^{1/2}\delta^{-O(\epsilon_1)-\epsilon}\#Y(\T),
		\end{equation}
		or
		$$\frac\lambda\delta\le C(\epsilon, \epsilon_1,\epsilon_2)K_3\delta^{-O(\epsilon_1)-\epsilon}(\frac{\#Y(\T)}{\sum_{T\in\T}\#Y(T)})^2\delta^{-t}.$$
		On the other hand, 
		$$\sum_{T\in\T}\#Y(T)\sim \frac{\lambda}{\delta}\#\T\lesssim \frac\lambda\delta\delta^{-t}.$$
		Combining the two estimates leads to 
		$$\sum_{T\in\T}\#Y(T)\le C(\epsilon, \epsilon_1,\epsilon_2)K_3\delta^{-O(\epsilon_1)-\epsilon}(\frac{\#Y(\T)}{\sum_{T\in\T}\#Y(T)})^2\delta^{-2t}.$$
		This proves our claim.

	\end{proof}
	The proof of Theorem \ref{maininct} will use induction on the scale parameter $\delta$. At each such scale, there is a dichotomy between the two-ends and the non-two-ends scenarios. In the first case, we get the desired estimate (in fact an even stronger one if $s\le d$) by invoking Theorem \ref{tWW25}. The key in dealing with the second case is that the product structure of tubes (similar to  the well-spaced condition, see Remark \ref{rwells}) forces them to behave well under rescaling.

	\subsection{Fourier decay of fractal measures}
	Let $\Gamma$ be the graph of a $C^3$ function $\gamma:[-1,1]\to\R$ satisfying the nonzero curvature condition $\min_{-1\le x\le 1}|\gamma''(x)|>0$. When $\gamma(\xi)=\xi^2$, we get the parabola $\Gamma=\PP^1$. The following was conjectured in \cite{O1} for $\Gamma=\PP^1$. There is growing supporting evidence for the case of arbitrary $\Gamma$ as above.
	\begin{conj}
		\label{hfrhvouuigrtiogupioythu09}
		Let $s\in [0,1]$. Let $\mu$ be a Borel measure supported on $\Gamma$
		satisfying 
		\begin{equation}
			\label{e4 giobjytibuiuiu}
			\mu(B(y,r))\lesssim r^s
		\end{equation}
		for each $y\in \R^2$ and each $r>0$. 
		Then  for each ball $B_R\subset\R^2$ of radius $R$  we have
		\begin{equation}
			\label{f[poi 0h0oi0-j9i9]}
			\|\widehat{\mu}\|^6_{L^6(B_R)}\les \begin{cases}R^{2-3s},\;s\le \frac12\\R^{1-s},\;s>\frac12\end{cases}.
		\end{equation}
	\end{conj}The paper \cite{O1} presented $R^{\epsilon}$ improvements (for some $\epsilon>0$) over the easier upper bound that follows using $L^4$ methods. That was achieved for each $s\in [0,1]$, using estimates for the Furstenberg set problem that were available at that time. This was an exciting connection that has since stimulated further research into the conjecture. In \cite{DD}, the $R^\epsilon$ improvement was recovered by combining a different Furstenberg set estimate with the Fourier analytic method of decoupling.

	The sharp estimate  was proved in the range $s\ge \frac23$, first in \cite{O2} for the parabola, then in \cite{Yi} for arbitrary curves. The conjecture appears to be significantly more difficult in the range $s\le \frac12$, as it entails certain square root cancellation.
	The bound for $s\le \frac12$
	\begin{equation}
		\label{kjfjrjhgturgurtupiotuhi}
		\|\widehat{\mu}\|^6_{L^6(B_R)}\les R^{2-3s+\beta},\;\;\beta=\frac{3s}{4}
	\end{equation}
	was proved in \cite{demwangszem} for arbitrary curves. 
	
	We prove two results for $s\le \frac23$. The first one is for the parabola. We write $\|\mu\|=\mu(\R^2)$.
	\begin{theorem}
		\label{io jfpurei90 h9=hhytj uy uythytj  }
		Let $s\le \frac23$. Assume $\mu$ is supported on $\PP^1$ and satisfies \eqref{e4 giobjytibuiuiu}. Then we have
		$$
		\int_{B_R}|\widehat{\mu}|^6\les R^{2-\frac{5s}{2}}\|\mu\|.
		$$
	\end{theorem}
	When $s=\frac23$, this recovers the sharp estimate in 
	\cite{O2}. When $s\le \frac12$, our result improves significantly \eqref{kjfjrjhgturgurtupiotuhi}
	in the case $\Gamma=\PP^1$. However, the two methods use Furstenberg set estimates in rather different ways. The argument in \cite{demwangszem} uses wave packet decomposition for $\widehat{\mu}$, in line with the earlier work \cite{DD}.  Wave packets are spatially localized on tubes that have special quasi-product structure. The $L^6$ norm is processed using $l^2$ decoupling, that reduces matters to counting heavy squares where many tubes intersect. Due to the product structure, this counting is done with the help of (case $s\le d$ of) Theorem \ref{maininct}.
	While this argument unfolds mostly on the spatial side (the domain of $\widehat{\mu}$), our proof of Theorem \ref{io jfpurei90 h9=hhytj uy uythytj  } takes place entirely on the spectral side (the domain of $\mu$). The inspiration for this argument comes from \cite{O1}. 
	
	To prove Theorem \ref{io jfpurei90 h9=hhytj uy uythytj  } we use virtually no Fourier analysis (e.g. decoupling). Instead, we take advantage of certain transformations that seem specific to the parabola, to recast the $L^6$ estimate into one that is purely of incidence geometric nature. In the  discrete case, this perspective goes back to \cite{BD}.
	In this new framework, we can no longer enforce quasi-product structure for tubes, but will instead gain some stronger non-concentration property for  squares. This allows us to take advantage of Theorem \ref{tWW25} to prove an incidence estimate of independent interest, Theorem \ref{huufwf[gmiu9i96=y]} (referred to in the introduction as Theorem \ref{thm.rectangles1}). We perform most of the analysis in a trilinear setting, that
	helps us enforce the much needed  non-concentration. The reduction to the linear setting follows via a broad-narrow analysis inspired by \cite{BoGu}. As a consequence of Theorem \ref{io jfpurei90 h9=hhytj uy uythytj  } we find the following upper bound for the $\delta$-discretized six-fold energy (sometimes referred to as 3-energy)
	$$\E_{3,\delta}(S)=\#\{(s_1,\ldots,s_6)\in S^6:\;|s_1+s_2+s_3-s_4-s_5-s_6|\lesssim \delta\}\les (\delta^{-s})^{7/2}$$
	for each $(\delta,s)$-KT set $S\subset \PP^1$, $s\le 2/3$. The exponent $7/2$ is only sharp for $s=2/3$. It is expected to be $3$ when $s\le 1/2$. The exponent $7/2$ matches the one in the discrete estimate 
	$$\E_3(S)=\#\{(s_1,\ldots,s_6)\in S^6:\;s_1+s_2+s_3-s_4-s_5-s_6=0\}\les (\#S)^{7/2},$$
	proved in \cite{BD} for arbitrary sets $S\subset\PP^1$.
	The proof in the discrete case involved a simple application of the Szemer\'edi--Trotter incidence bound for lines and points. In the world of tubes and squares, a bound of analogous strength only holds for families with special properties (e.g. such as those in Theorems \ref{maininct} and \ref{tWW25}). Our arguments invest significant effort into identifying a new such family that fit the context of our application. See Theorem \ref{huufwf[gmiu9i96=y]}. 
	
	\smallskip
	
	Here is our second result, for arbitrary curves. We use the method in \cite{demwangszem}, in fact our new result for $s\ge \frac12$ complements \eqref{kjfjrjhgturgurtupiotuhi}. More precisely, we combine decoupling with our new incidence estimate for quasi-product sets of tubes, Theorem \ref{maininct} for $s>d$, in the form of Corollary \ref{pout98u8huy8ht-8h}.  
	
	\begin{theorem}
		\label{jki regu rthu8 yuh 60-u4j}
		Let $\frac12\le s\le 1$. Assume $\mu$ is supported on a curve $\Gamma$ with nonzero curvature, and satisfies \eqref{e4 giobjytibuiuiu}.
		Then
		$$
		\|\widehat{\mu}\|_{L^6(B_R)}^6\les R^{1-s+\beta},\;\;\;\beta=\frac{1-s^2}{2s+1}.
		$$
		If $\mu$ is AD-regular, that is  
		$$C_\mu^{-1}r^s\le \mu(B(y,r))\le C_\mu r^s,\;\;\forall y\in\operatorname{supp}(\mu),$$
		then
		$$\|\widehat{\mu}\|_{L^6(B_R)}^6\les R^{1-s+\beta-c_\mu},\;\;\;\beta=\frac{1-s}{2},$$
		when $c_\mu>0$ depends on $C_\mu$.
	\end{theorem}

	\subsection{Sum-product problems}
	The discretised sum-product problem asks given $0 < s \leq t < 1, C >0$  what is the largest value of $c  = c(C,s,t) >0$ so that for all $(\delta,s,C)$-sets with $\# A= \de^{-t}$ do we have
	\begin{equation}\label{eq.dissumprodfsdf}
		\cn{A+A} + \cn{AA} > \delta^{-c}\#A,
	\end{equation}
	provided that $\de > 0$ is small enough. This problem is a `fractal' analogue of the Erd\H os--Szemer\'edi conjecture \cite{erdsz}, and was asked in \cite{kat}, in relation to the Erd\H os--Volkmann ring problem \cite{erd}. This problem has been explored in many works, with non-quantitative estimates on $c > 0$ given in \cite{bou03}, \cite{bougam}, \cite{bou}, \cite{sumprodent}. When $s = t,$ many quantitative estimates for $c$ have been discovered, see \cite{gut}, \cite{che}, \cite{furen}, \cite{orpshabc}, \cite{mator}, \cite{renwang} to name a few, with an improvement when $A$ has additional regularity assumptions given in \cite{orreg}.
	
	A direct consequence of \eqref{eq.dissumprodfsdf}is for $A \subset \R$ with $\dimh A = s$ we have
	\begin{equation}
		\lbox ((A+A) \times AA)) \geq 2s + c.
	\end{equation}
	This strictly implies, for example, using product formulae, that
	\begin{equation}\label{eq.lopsided}
		\max\{\ubox (A+A), \lbox( AA)\} \geq s + c/2.    \end{equation}
	One may then reasonably ask what are the best bounds that one can obtain for \eqref{eq.lopsided}. Using Theorem \ref{thm.product} we obtain an improvement a discretised version of \eqref{eq.lopsided}. Due to the geometric meaning of $A/A,$ our arguments find it convenient to replace the product-set $AA$ with the ratio-set $A/A$.
	\begin{theorem}\label{thm.dis}
		Let $0 < s \leq 136/265.$ There exists $C,\delta_0, \epsilon> 0$ so that the following holds for all $0 < \delta < \delta_0.$
		
		\indent Let $A \subset [1/2,1]$ be a $\de$-separated $(\delta, s, \de^{-\e})$-set. Then one of the following must hold.
		\begin{equation}
			\cns{\rho}{A+A} > \rho^{-43s/34} \text{ for some } \delta < \rho < \delta^{C\e},
		\end{equation}
		\begin{equation}
			\cn{A/A} > \delta^{-43s/34+C\e}.
		\end{equation}
	\end{theorem}
	The method allows the sum-set to be replaced by the difference-set, but to our knowledge, not the ratio-set with the product-set.
	
	The previous best bound for this problem replaced $43/34$ with the smaller exponent $5/4,$ and is obtained by applying the argument of Elekes \cite{elekes1997number} with the full resolution of Furstenberg problem, due to Ren--Wang \cite{renwang}. The details of this argument may be found at \cite[Corollary 6.6]{orpshabc}. However their conclusion is stronger, in the sense that they obtain
	\begin{equation}
		\cn{A+A}\cn{A/A} > \delta^{-5s/2},
	\end{equation}
	which, as discussed, implies, but is not implied by, the form of our result. 
	
	The idea of our proof is combine the aforementioned argument of Elekes, with the beautiful argument of Solymosi \cite{soly14/11} where for a finite subset $A \subset \mathbb{C}$ it is shown that 
	\begin{equation}
		\#(A+A)^8\#(A/A)^3 \gtrsim \#A^{14}/\log^3\#A.
	\end{equation}
	In particular, this gives
	\begin{equation}
		\#(A+A) + \#(A/A) \gtrsim \#A^{14/11}/ \log^{3/11}\#A.
	\end{equation}
	
	Unfortunately, we were not able to fully adapt Solymosi's proof, and the reason  we obtain $5/4 < 43/34 < 14/11$ is due to the fact that we essentially have to take some form of average of the two methods.

	A fairly straightforward corollary of Theorem \ref{thm.dis} is the following in terms of fractal dimensions.
	\begin{theorem}\label{thm.sumprod}
		Let $0 < s \leq 136/265.$  For all $A \subset \R$ with $\dimh A = s$ we have
		\begin{equation}
			\max\{\ubox(A+A),\lbox(A/A)\} \geq 43s/34.
		\end{equation}
	\end{theorem}
	
	\hfill

	A weaker version of the discretised sum-product problem would be to assume good additive (for example) properties of $A,$ and ask just about the growth of $AA$ or $A/A.$ 
	
	Applying Theorem \ref{thm.product} in a different way, we also obtain the following result which gives a sharp result when the sum-set is small.
	
	\begin{theorem}\label{thm.fewsums}
		Let $0 \leq s \leq 2/3.$ There exists $C,\delta_0, \epsilon> 0$ so that the following holds for all $0 < \delta < \delta_0.$ 
		
		Let $A \subset [1/2,1]$ be a $\delta$-separated 
		$(\delta,s,\de^{-\e})$-set with $\#A < \de^{-s -\e}. $ If $s \leq 1/2,$ then
		\begin{equation}
			N_\delta(A+A)^6N_\delta(A/A)> \de^{C\e}\# A^8.
		\end{equation}
		If $s>1/2,$ then
		\begin{equation}
			N_\delta(A+A)^6N_\delta(A/A) > \de^{C\e}\de^{-2}\# A^4.
		\end{equation}
		In particular, for $0 < c_0 < s,$ when $\cn{A+A} \leq \de^{-c_0}\#A,$ we have
		\begin{equation}\label{eq.fewsumssection1}
			\cn{A/A} > \de^{6c_0 + C\e}\min\{\#A^2, \de^{-2}\#A^{-2}\}.
		\end{equation}
	\end{theorem}
	The method allows the sum-set to be replaced by the difference set, but to our knowledge, not the ratio-set with the product-set.
	
	When $0 < s \leq 1/2,$ this improves the result of \cite{renwang} when $0 \leq c_0 \leq s/14,$ which implies for $0 < s \leq 2/3$
	\begin{equation}\label{eq.renwangfewsums}
		\cn{A/A} \gtrsim \de^{c_0 + O(\e)}\# A^{3/2}.
	\end{equation}
	When $1/2 \leq s \leq  2/3$ \eqref{eq.fewsumssection1} becomes poorer as $s$ gets larger, in particular, we see no growth at $s = 2/3.$ However, when $1/2 \leq s \leq 4/7,$ and $0 \leq c_0 \leq \frac{4-7s}{10},$ we are able to improve \eqref{eq.renwangfewsums}.
	
	We do not obtain any results when $s > 2/3.$ However, when $2/3 < s < 1,$ it was obtained in \cite{furen} that
	\begin{align}
		\cn{A+A}\cn{A/A}  &\gtrsim \de^{-1 + O(\e)} \# A \text{ if } 2/3 \leq s < 1 
	\end{align}
	which is sharp -  see \cite[Proposition 4.2]{mator}.

	We follow the idea of Elekes--Rusza \cite{elekesruzsafewsums}, who use Szemer\'edi--Trotter in a clever way to obtain for a discrete $A \subset \R$ that 
	\begin{equation}
		\#(A+A)^4\# (A/A) \gtrsim \# A^6 / \log \#A.
	\end{equation}
	We were only able to replace $4$ by $6$ in the exponent of the sum-set. Nor were we able to exactly mimic their proof with Theorem \ref{maininct} in hand, and the consideration of the problem at an intermediate scale is necessary. See Section \ref{sect.fewsums} for the proof.
	
	\subsection*{Acknowledgements}
	The authors met at the Reading Groups in Analysis  program that took place in August 2025 at the University of Pennsylvania. They are grateful to the organizers for putting together a wonderful two-week event.
	
	W.O.R.~would like to thank Pablo Shmerkin and Joshua Zahl for all of their advice and support. Further thanks is extended to Hong Wang, for hosting his visit to NYU and providing ever-insightful and motivational discussions.
	
	The authors extend further thanks to Tuomas Orponen, Hong Wang, Shukun Wu and Joshua Zahl  for many interesting conversations on the topics of the current paper. 
	\section{Auxiliary results}
	Fix $\alpha>1$.	
	Let $I(\delta)=I_{s,d,\alpha}(\delta)$ be the smallest constant such that
	\begin{equation}
		\label{jrjugitgih9iy9hih}
		\sum_{T\in\T}\#Y(T)\le I(\delta)K_3^{\frac13}{\delta^{\frac{-2s-2d}\alpha}}\#Y(\T)^{1-\frac1\alpha}
	\end{equation}
	for each $K_3\ge 1$, for each $(\delta,s,d,1,1)$-quasi-product set $\T$ and  each shading $Y$ of $\T$ such that $Y(T)$ is a $(\delta,\sigma,K_3)$-KT set for each $T\in\T$. The reader should keep in mind that $\#\T\lesssim \delta^{-s-d}$.
	
	The next result explores the dependence on constants if the triple $(1,1,K_3)$ is replaced with $(K_1,K_2,K_3)$. 
	\begin{prop}
		\label{lkogjthuyuh09yu}\label{prop.product}
		Assume $\T$ is a $(\delta,s,d,K_1,K_2)$-quasi-product set. Consider a shading $Y$ of $\T$ such that $Y(T)$ is a $(\delta,\sigma,K_3)$-KT set for each $T\in\T$.
		
		Then  we have
		$$\sum_{T\in\T}\#Y(T)\les I(\delta) K_3^{\frac13}(K_1K_2)^{1-\frac1\alpha} (\delta^{-s-d}\#\T)^{1/{\alpha}}\#Y(\T)^{1-\frac1\alpha}.$$	
	\end{prop}
	\begin{proof}
		Step 1.	
		\\
		\\	
		We first prove the inequality
		\begin{equation}
			\label{hhufrufruguguitgyy986}
			\sum_{T\in\T}\#Y(T)\les I(\delta)K_3^{\frac13}(\delta^{-s-d}\#\T)^{1/{\alpha}}\#Y(\T)^{1-\frac1\alpha}
		\end{equation}
		for $(\delta,s,d,1,1)$-quasi-product sets $\T$ such that $Y(T)$ is a $(\delta,\sigma,K_3)$-KT set for each $T\in\T$.
		
		We apply the reasoning in Step 2 of the proof of Proposition 4.3 in \cite{demwangszem}. At the expense of $\delta^{-\epsilon}$	losses, we may assume $\T$ satisfies all the uniformity assumptions needed for the Corollary 3.3 in \cite{demwangszem} to be applicable. This provides $N\sim \frac{\delta^{-s-d}}{\#\T}$ many rigid motions $\cA_i$ of the plane such that the collection
		$$\T_{new}=\bigcup_{i=1}^N\T_i,\;\;\T_i=\cA_i(\T)$$
		has the following properties:
		
		(P1) $\T_{new}$ is a $(\delta,s,d,K_1,K_2)$-quasi-product set for some $K_1,K_2\les 1$.
		
		(P2)  each tube belongs to at most $\les 1$ many $\T_i$. In particular, $\#\T_{new}\approx \delta^{-s-d}$. 
		By incurring $\delta^{-\epsilon}$-losses, we may assume $K_1,K_2=1$, $\# \T_{new}\lesssim \delta^{-1}$. 
		
		For each $T_i\in\T_i$ with $T_i=\cA_i(T)$ ($T\in\T$) we define $Y(T_i)=\cA_i(Y(T))$.
		
		Given $\PP$, let $\PP_i=\cA_i(\PP)$ and note that for each $i$
		$$\sum_{T\in\T}\#Y(T)=\sum_{T_i\in\T_i}\#Y(T_i).$$
		Using (P2) it is immediate that
		$$\sum_{i=1}^N\sum_{T_i\in\T_i}\#Y(T_i)\les \sum_{T\in \T_{new}}\#Y(T).$$ 
		Using (P1) and the definition of $I(\delta)$ we have 
		$$\sum_{T\in \T_{new}}\#Y(T)\le I(\delta)K_3^{\frac13}\delta^{\frac{-2s-2d}\alpha}\#Y(\T_{new})^{1-\frac1\alpha}.$$
		Combining the last three estimates and the fact that $\#Y(\T_{new})\le N\#Y(\T)$ we conclude that
		$$\sum_{T\in\T}\#Y(T)\les\frac1{N} I(\delta)K_3^{\frac13}\delta^{\frac{-2s-2d}\alpha}\#Y(\T)^{1-\frac1\alpha}N^{1-\frac1\alpha}=I(\delta)K_3^{\frac13}(\delta^{-s-d}\#\T)^{1/{\alpha}}\#Y(\T)^{1-\frac1\alpha}.$$
		
		\smallskip
		Step 2.
		\\
		\\
		Let now $\T$ be an arbitrary $(\delta,s,d,K_1,K_2)$-quasi-product set.	
		We  combine Lemma 2.8 and Lemma 2.9 in \cite{demwangszem} to partition $\T$ into $M\les K_1K_2$
		many $(\delta,s,d,1,1)$-quasi-product sets $\T^j$. 
		
		Noting that
		$$\sum_{T\in\T}\#Y(T)=\sum_{j=1}^M\sum_{T\in\T^j}\#Y(T)$$
		and using \eqref{hhufrufruguguitgyy986} for each term followed by H\"older's inequality in $j$,  we finish the proof
		\begin{align*}
			\sum_{T\in\T}\#Y(T)&\les \sum_{j=1}^MI(\delta)K_3^{\frac13}(\delta^{-s-d}\#\T^j)^{1/{\alpha}}\#Y^i(\T^j)^{1-\frac1\alpha}\\&\le  I(\delta)K_3^{\frac13}M^{1-\frac1\alpha}(\delta^{-s-d}\#\T)^{1/{\alpha}}\#Y(\T)^{1-\frac1\alpha}\\&\les I(\delta) K_3^{\frac13}(K_1K_2)^{1-\frac1\alpha}(\delta^{-s-d}\#\T)^{1/{\alpha}}\#Y(\T)^{1-\frac1\alpha}.
		\end{align*}

	\end{proof}

	The next result shows the hereditary nature of quasi-product sets. This property is essentially preserved at smaller scales, with quantifiable losses that fit well into the induction on scales that will be used in the proof of  Theorem \ref{maininct}. Well-spaced families of tubes also display a similar behavior, see Remark \ref{rwells}. However, arbitrary $(\delta,t,1)$-KT sets typically fall outside of this paradigm. This explains why Theorem \ref{maininct} is false at this level of generality. See e.g. the train-track example in \cite{demwangszem} for $t=1$, and \cite[Section 2]{furen}.
	\begin{prop}
		\label{fkjgu8rtu9i6iy-}
		Let $\T$ be a $(\delta,s,d,1,1)$-quasi-product set.
		
		Fix $\delta\le \Delta\le 1$, and a $\Delta$-square $Q$ inside $[0,1]^2$. Fix $M\le \Delta^{-s}$.
		
		Let $\cT_Q$ be the collection of $\delta\times \Delta$-segments $\tau$ inside $Q$ that lie at the intersection of $\sim M$ many tubes $T\in\T$. Then the $\times \Delta^{-1}$ rescaling of $\cT_Q$ is a union of $\les 1$ many $(\delta/\Delta,s,d,K_1,K_2)$-quasi-product sets of $\delta/\Delta$-tubes in $[0,1]^2$, with parameters $K_1,K_2$ satisfying $K_1K_2\lesssim \frac1{M\Delta^s}$.
	\end{prop}
	\begin{proof}
		Step 1. (pigeonholing the parameter $N$)
		Let $\Lambda$ be the direction set of $\T$. For each $\tau\in\cT_Q$, the directions of the $\sim M$ tubes $T\in\T$ containing $\tau$ lie inside an arc $\theta_\tau$ of length $\sim \delta/\Delta$. By pigeonholing we may assume 
		$$\#(\theta_\tau\cap \Lambda)\sim N$$
		for each $\tau$, for some dyadic value $N\ge M$. There are $O(\log\delta^{-1})$ many such values, and there will be one collection of tubes for each $N$. We continue to call $\cT_Q$ the subcollection corresponding to the fixed $N$.
		\\
		\\
		Step 2. (spacing properties of the direction set for $\cT_Q$)
		For each arc $\gamma$ with length $|\gamma|\gtrsim \delta/\Delta$ we have
		$$N\#\{\text{distinct }\theta_\tau\subset \gamma:\;\tau\in\cT_Q\}\lesssim \#(\Lambda\cap \gamma)\lesssim (\frac{|\gamma|}\delta)^s.$$
		This shows that the direction set $\Lambda_Q$ for $\cT_Q$ is a $(\frac{\delta}{\Delta},s,K_1)$-KT set, with $K_1\sim \frac1{N\Delta^s}$.
		\\
		\\
		Step 3. (spacing properties in a fixed direction)
		
		Fix a direction in $\Lambda_Q$, corresponding to an arc $\theta\subset \cS^1$ of length $\sim \delta/\Delta$. Given  $\frac{\delta}{\Delta}\le r\le 1$,  fix a rectangle $R\subset Q$ with dimensions $r\Delta\times \Delta$, and long side in the  direction $\theta^\perp$. We count the number $L$ of $\tau\in\cT_Q$ in this direction, lying inside $R$. There are $\sim LM$ tubes $T\in\T$ that point in one of the $\sim N$ directions in $\theta\cap \Lambda$. For each of these directions, by virtue of intersecting $R$, the tubes $T$ in that direction will lie inside a fixed $100r\Delta\times 1$ box that contains  $R$. Using that $\T$ is a  $(\delta,s,d,1,1)$-set, it follows that
		$$LM\lesssim N(\frac{r\Delta}{\delta})^{d},$$
		or 
		$$L\lesssim (\frac{r}{\delta\Delta^{-1}})^{d}\frac{N}{M}.$$   
		This shows that the rescaled version of $\cT_Q$ has the following property: tubes in a given direction $\theta$ form a $(\frac{\delta}{\Delta},d,K_2)$-KT set, with $K_2\sim \frac{N}{M}$.
		
	\end{proof}
	The argument in Proposition \ref{lkogjthuyuh09yu} gives the following general version of \eqref{jefhureyg7yg78yg78y}, which will be used later.  
	\begin{theorem}
		\label{completeWW}
		Let $\T$ be a $(\delta,t,K_1)$-KT set for some $t\in(0,2)$, and let $0<\epsilon_2<\epsilon_1$.
		Let $Y$ be a shading of $\T$.
		Assume that each $Y(T)$ is a $(\delta,\sigma,K_2)$-KT set with cardinality $\sim N$,  and assume that it is also $(\epsilon_1,\epsilon_2)$-two-ends.
		
		Then for each $\epsilon>0$
		$$   N^{1/2}\delta^{t/2}\sum_{T\in\T}\#Y(T)\le C(\epsilon, \epsilon_1,\epsilon_2)K_1K_2^{1/2}\delta^{-O(\epsilon_1)-\epsilon}\#Y(\T).
		$$
	\end{theorem}
	
	\subsection{Properties of discretised sets} In this subsection we record some basic properties. Often, we will refer to the results recorded here without reference.

	The below definition is standard, and can be found at \cite[Definition 2.3]{demwangszem}, for example.
	\begin{definition}\label{def.uniform}
		Let $\epsilon >0.$ Let $T_\epsilon$ satisfy $T_\epsilon^{-1}\log (2T_\epsilon) = \epsilon.$ Given $0 < \delta \leq 2^{-T_\epsilon},$ let $m$ be the largest integer such that $mT_\epsilon \leq \log (1/\delta).$ Set $T := \log(1/\delta)/m.$ Note that $\delta = 2^{-mT}.$ We say that a $\delta$-separated $A \subset \R^d$ is \textit{$\e$-uniform} if for each $\rho = 2^{-jT}, 0 \leq j \leq m$ and each $P,Q \in \cD_\rho(A),$ we have
		\begin{equation}
			\#(A \cap P) \sim \#(A \cap Q).
		\end{equation}
		A similar definition may be given if $A$ is a collection of $\delta$-balls.
	\end{definition}
	Below is \cite[Lemma 2.4]{demwangszem}.
	\begin{lemma}
		Let $\epsilon > 0.$ Let $A \subset \R^d$ be a $\delta$-separated $\e$-uniform set. For each $\delta \leq \rho \leq 1$ and each $P,Q \in \cD_\rho(A)$ we have 
		\begin{equation}
			C_\epsilon^{-1}\#(A \cap P) \leq \#(A \cap Q) \leq C_\e \#(A \cap P).
		\end{equation}
	\end{lemma}
	Below is \cite[Lemma 2.15]{orpshabc}.
	\begin{lemma}\label{lem.unifsubset}
		Let $\epsilon > 0$ and suppose that $\delta >0 $ is small enough in terms of $\e.$ Let $A \subset \R^d$ be a $\delta$-separated set. Then there is an $\e$-uniform $A_0 \subset A$ with $\#A_0 \gtrsim \delta^{\epsilon}\#A.$ 
	\end{lemma}

	The proposition below records the relation between $(\delta,s)$-sets and $(\delta,s)$-KT sets, and explores how properties at fine scales influence properties at courser scales, as well as how properties change when subject to a rescaling. 
	\begin{proposition}\label{prop.properties}
		Let $0 \leq s \leq d, C, \delta,\epsilon  >0.$ Let $\delta \leq \rho.$ Let $A \subset \R^d$ be an $\epsilon$-uniform collection of $\delta$-balls. 
		
		\begin{enumerate}
			\item  If $A$ is a $(\delta,s,C)$-set then $A$ is a $(\delta,s,C\#A \delta^s)$-KT set.
			\item If $A$ is a $(\delta,s,C)$-KT set then $A$ is a $(\delta,s, C\delta^{-s}/\#A)$-set.
			\item If $A$ is a $(\delta,s,C)$-set then $A_\rho$ is a $(\rho,s,C_\e C)$-set. 
			\item If $A$ is a $(\delta,s,C)$-KT set then $A_\rho$ is a $(\rho, s, C_\epsilon C (\rho/\delta)^s \cns\rho A/\#A)$-KT set.
			\item If $A$ is a $(\delta,s,C)$-KT set then the $\times \rho^{-1}$ dilate of $A \cap B(\rho)$ is a 
			$(\delta/\rho,s,C)$-KT set, where $B(\rho)$ is any ball of radius $\rho.$
		\end{enumerate}
		\begin{proof}
			Let $x \in \R^d, \delta \leq r, \rho \leq R.$
			\begin{enumerate}
				\item This follows from the observation that
				\begin{equation}
					\#(A \cap B(x,r)) \leq Cr^s\#A = C\#A\delta^s(r/\delta)^s.
				\end{equation}
				\item This follows from the observation that
				\begin{equation}
					\#(A \cap B(x,r)) \leq (C(r/\delta)^s = C\delta^{-s}/\#A) r^s \#A.
				\end{equation}
				\item The first statement is \cite[Lemma 2.15]{orpshabc}.
				\item This follows from (1), (2), and (3).
				\item This follows from the observation that if $\delta/\rho \leq r \leq 1$ then
				\begin{equation}
					\#(A \cap B(\rho) \cap B(x,r\rho)) \leq C(r/(\delta/\rho))^s,
				\end{equation}
				and this quantity is invariant under rescaling. 
			\end{enumerate}
		\end{proof}
		
	\end{proposition}
	Below is \cite[Lemma 3.13]{fassorp}.
	\begin{lemma}\label{lem.disfrost}
		Let $0 \leq s \leq d, \delta,\kappa  >0.$  Let $A \subset \R^d$ and suppose that $\cont{s}(A) > \kappa.$ Then $A$ contains a $(\delta,s)$-set with cardinality  $\gtrsim \kappa\delta^{-s}.$
	\end{lemma}
	\begin{lemma}\label{lem.productset}
		Let $0 \leq s \leq 1, C, \de >0.$ Let $A \subset \R$ be a $(\delta,s,C)$-KT set. Let $T$ be a $\delta$-tube. Then $(A \times A) \cap T$ is a $(\delta,s,O(C))$-KT set.
		\begin{proof}
			Let $r > \de$ and let $B$ be a ball of radius $r.$ Without loss of generality suppose the tube makes an angle more than 45 degrees with the ordinate. By Fubini, and writing $\pi$ for the projection to the abscissa, we have
			\begin{align}
				\#((A \times A) \cap T\cap B) &= \sum_{x \in A \cap \pi(B)}  \#(\pi^{-1}(x)\cap (A \times A \cap T)) \\
				&\lesssim  \#(A \cap \pi(B))\\
				&\leq C(r/\delta)^s,
			\end{align}
			as required.
		\end{proof}
	\end{lemma}

	\section{Estimates for the number of $r$-heavy squares}\label{sect.inc}
	We first prove Theorem \ref{maininct} and then record some sharp consequences for the number of $r$-heavy squares. 
	\begin{proof}[Proof of Theorem \ref{maininct}]
		Recall the definition of $I(\delta)$ in   \eqref{jrjugitgih9iy9hih}. In light of Proposition \ref{lkogjthuyuh09yu}, our goal is to verify that $I(\delta)\les 1$. This will follow immediately by iterating the inequality
		\begin{equation}
			\label{figjiortugre-o0y0}
			I(\delta)\les I(\delta^{1-\epsilon_1})\delta^{-\epsilon_2}+C(\epsilon_1,\epsilon_2)\delta^{-O(\epsilon_1)},
		\end{equation}
		with the choice, say $\epsilon_1=\epsilon$, $\epsilon_2=\epsilon^3$.

		It remains to prove \eqref{figjiortugre-o0y0}.
		Fix $\epsilon_1,\epsilon_2$.  We apply results from the previous section to $\Delta=\delta^{\epsilon_1}$. Fix a $(\delta,s,d,1,1)$-quasi-product set $\T$ and a shading $Y$ such that each $Y(T)$ is a $(\delta,\sigma,K_3)$-KT set.
		
		Split $[0,1]^2$ into $\Delta$-squares $Q$. Let $\cT_{Q,all}$ be the collection of all $\Delta\times \delta$ tubes $\tau$ inside $Q$. There are roughly $(\Delta/\delta)^2$ many 
		such tubes. 
		Given $T\in\T$ and $Q$ we call $\tau\in\cT_Q$ a $\Delta$-segment of $T$ if the angle between the directions of $T$ and $\tau$ is $\lesssim \frac{\delta}{\Delta}$, and if $T\cap \tau\not=\emptyset$. Each $T$ has $\sim 1$ many $\Delta$-segments inside each $Q$ that $T$ intersects.
		
		By considering two categories separately, we may assume one of the two (mutually exclusive) scenarios holds.
		\\
		\\
		(TE) $Y(T)$ is $(\epsilon_1,\epsilon_2)$-two-ends for each $T,$
		\\
		\\
		(NTE) each $T$ has a $\Delta$-segment $\tau_T$ that contains more than a $\delta^{\epsilon_2}$-fraction of $Y(T)$. If there are multiple such segments, we pick an arbitrary one. 
		
		In the (TE) case we get the bound 
		$$I(\delta)\le C(\epsilon, \epsilon_1,\epsilon_2)\delta^{-O(\epsilon_1)-\epsilon}$$
		by Theorem  \ref{tWW25}. We just need to verify that
		$$ \delta^{\frac{-2s-2d}{3}}\#Y(\T)^{1-\frac13}\lesssim  \delta^{\frac{-2s-2d}{\alpha}}\#Y(\T)^{1-\frac1\alpha},$$
		when $\alpha<3$. This is equivalent with proving the bound
		$$\#Y(\T)\lesssim \delta^{-2s-2d}.$$ The only nontrivial case is when $s+d<1$. Note that $\sigma=s+d$, and thus $\#Y(T)\lesssim \delta^{-s-d}$ for each $T\in\T$. Since $\#\T\lesssim \delta^{-s-d}$, the result follows.
		\medskip
		
		In the (NTE) case, let $\cT$ be the collection of all distinct segments $\tau_T\in\cup_Q\cT_{Q,all}$. Pigeonholing we may assume each $\tau\in\cT$ coincides with $\tau_T$ for $\sim M$ many $T\in\T$.
		We may also rename $\T$ to consist  only of those $T$ such that $\tau_T=\tau$ for some $\tau\in\cT$.
		Double counting shows that
		$$\#\cT\sim \frac{\#\T}{M}.$$
		
		Let $\cT_Q$ be those $\tau\in\cT$ lying inside $Q$. We define the shading $Y_Q(\tau)$
		as follows. Pick any $T\in \T$ such that $\tau=\tau_T$, and let $Y_Q(\tau)=Y(T)\cap \tau$. Note that the $\times \Delta^{-1}$ rescaling of this shading is a $(\delta/\Delta,\sigma,K_3)$-KT set.
		Propositions \ref{lkogjthuyuh09yu} and \ref{fkjgu8rtu9i6iy-} combine to show that  
		$$\sum_{\tau\in\cT_Q}\#Y_Q(\tau)\lesssim  I(\frac\delta\Delta)K_3^{\frac13}(\frac1{M\Delta^s})^{1-\frac1{\alpha}}[\;(\frac\delta\Delta)^{-s-d}\#\cT_Q\;]^{\frac{1}{\alpha}}\#Y_Q(\cT_Q)^{1-\frac1{\alpha}}.$$

		We use that
		$$\sum_{T\in\T}\#Y(T)\lesssim \delta^{-\epsilon_2}M\sum_{Q}\sum_{\tau\in\cT_Q}\#Y_Q(\tau)$$
		and
		$$\sum_{Q}\#Y_Q(\cT_Q)\le \#Y(\T),\;\;\text{since the cubes }Q\text{ are pairwise disjoint}.$$
		Combining all these with H\"older's inequality we find
		\begin{align*}
			\sum_{T\in\T}\#Y(T)&\lesssim \delta^{-\epsilon_2}I(\delta/\Delta)K_3^{\frac13}M(\frac{\delta}{\Delta})^{-\frac{s+d}{\alpha}}(\frac1{M\Delta^s})^{1-\frac1{\alpha}}(\sum_Q\#\cT_Q)^{\frac{1}{\alpha}}(\sum_{Q}\#Y_Q(\cT_Q))^{1-\frac1{\alpha}}\\&\lesssim \delta^{-\epsilon_2}I(\delta/\Delta)K_3^{\frac13}M(\frac{\delta}{\Delta})^{-\frac{s+d}{\alpha}}(\frac1{M\Delta^s})^{1-\frac1{\alpha}}(\frac1{M})^{\frac1\alpha}\#\T^{\frac{1}{\alpha}}\#Y(\T)^{1-\frac1{\alpha}}
			\\&\lesssim \delta^{-\epsilon_2}I(\delta/\Delta)K_3^{\frac13}\Delta^{\frac{2s+d}\alpha-s}\delta^{-\frac{2s+2d}{\alpha}}\#Y(\T)^{1-\frac1\alpha}.
		\end{align*}
		The exponent of $\Delta$ is positive if $s<d$, and is zero when $s\ge d$.
		We thus have 
		$$I(\delta)\les \delta^{-\epsilon_2}I(\delta/\Delta).$$	
	\end{proof}	
	We now present some immediate consequences for the number of $r$-heavy squares. 
	Let $$\cP_r(\T)=\{p:\;\#\{T\in\T:\;p\cap T\not=\emptyset\}\sim r\}.$$
	\begin{cor}\label{pout98u8huy8ht-8h} 
		Assume $s+d=1$. Then for each $(\delta,s,d,K_1,K_2)$-quasi-product set $\T$ and each $r\ge 1$ we have 
		\begin{equation}
			\label{fjiugutgut hu8uh8-8}
			\#\cP_r(\T)\les (K_1K_2)^{1/s}\frac{\delta^{-1}\#\T}{r^{\frac{s+1}{s}}}
		\end{equation}
		if $s>d$ and 
		\begin{equation}
			\label{fjiugutgut hu8uh8-88}
			\#\cP_r(\T)\les (K_1K_2)^2\frac{\delta^{-1}\#\T}{r^{3}}
		\end{equation}
		if $s\le d$.	
	\end{cor}
	\begin{proof}We apply Theorem \ref{maininct}  to $\T$ and the shading $Y(T)=\{p\in\cP_r(\T):\;p\cap T\not=\emptyset\}$.
		We first note that since $\sigma=1$, the non-concentration assumption on $Y(T)$ is trivially satisfied with $K_3=1$. Since $Y(\T)=\cP_r(\T)$ and
		$$\#\cP_r(\T)r\sim \sum_{T\in\T}\#Y(T),$$Theorem \ref{maininct}  gives
		$$\#\cP_r(\T)r\les (K_1K_2)^{1-\frac1\alpha}(\delta^{-1}\#\T)^{1/\alpha} \#\cP_r(\T)^{1-\frac1\alpha},$$
		or
		$$\#\cP_r(\T)\les (K_1K_2)^{\alpha-1}\frac{\delta^{-1}\#\T}{r^\alpha}.$$	
		Since $\alpha=\min\{3,\frac{s+1}{s}\}$, the result follows.

	\end{proof}
	
	\begin{remark}
		We note that \eqref{fjiugutgut hu8uh8-88} was proved in \cite{demwangszem}. The argument in \cite{demwangszem} gives no clue on what the correct analogue of \eqref{fjiugutgut hu8uh8-88} should be when $s>d$. The new perspective in the current paper reveals that this analogue is \eqref{fjiugutgut hu8uh8-8}. 
		
		Assume $K_1,K_2\sim 1$.
		As pointed out in \cite{demwangszem}, \eqref{fjiugutgut hu8uh8-88} is sharp for all $r$.
		The following example shows the sharpness of \eqref{fjiugutgut hu8uh8-8}, even in the AD-regular case. 	
		Start with a $(\delta,1-s)$-set $\cP\subset \cD_\delta$ of squares intersecting $[0,1]\times \{0\}$, with cardinality $\delta^{s-1}$. Consider also an AD-regular set $\Lambda$. By this we mean that each $\rho$-arc ($\rho>\delta$) centered at some $\theta\in\Lambda$ intersects $\sim (\rho/\delta)^s$ many points in $\Lambda$.  Finally, for each $p\in \cP$, let $\T_p$  be the collection of $\delta$-tubes passing through $p$, with one tube for each direction in $\Lambda$. Note that $\T=\cup_{p\in \cP}\T_p$ is a $(\delta,s,1-s,O(1),O(1))$-quasi-product set. An easy computation shows that for each $p\in\cP$ and each $1\le r\lesssim \delta^{-s}$ we have $\#\cP_r(\T_p)\sim \delta^{-1-s}r^{-\frac{s+1}{s}}$. The set $\cP_r(\T_p)$ lies inside an annulus of thickness $\sim r^{-1/s}$ centered at $p$.
		
		It is easy to see that the sets $(\cP_r(\T_p))_{p\in\cP}$ are pairwise disjoint if $r\gg \delta^{s(s-1)}$. We can show that they (or at least their subsets lying at height $\sim r^{-1/s}$) are essentially disjoint for each $r$. We claim that there is a subset $\cP'\subset \cup_{p\in\cP}\cP_r(\T_p)$ with size
		$$\#\cP'\approx (\#\cP)\delta^{-1-s}r^{-\frac{s+1}{s}}\sim \frac{\delta^{-2}}{r^{\frac{s+1}{s}}}$$
		so that each $q\in \cP'$ is in $\les 1$ many sets $\cP_r(\T_p)$. 
		
		Let $\cP'$ be those $p\in\cD_\delta$ that belong to $\sim M$ many $\cP_r(\T_p)$ (for some $M$) and such that 
		$$M\#\cP'\approx \frac{\delta^{-2}}{r^{\frac{s+1}{s}}}.$$
		The existence of such $M$ follows via pigeonholing.
		Since $\cP'\subset \cP_{rM}$, \eqref{fjiugutgut hu8uh8-8} implies that
		$$\#\cP'\les \frac{\delta^{-2}}{(rM)^{\frac{s+1}{s}}}.$$
		These imply $M\approx 1$, and the claim follows.
	\end{remark}
	\begin{remark}The dependence of $K_1,K_2$ in \eqref{fjiugutgut hu8uh8-8}, \eqref{fjiugutgut hu8uh8-88} is sharp for each $K_1K_2\lesssim r\le K_1\delta^{-s}$, when $K_1\lesssim K_2^{\frac{1-s}{s}}$. To see this, define 
		$\rho=\delta K_2^{\frac1s}$
		and $r'=\frac{r}{K_1K_2}$. Let $\T_\rho$ be $(\rho,s,1-s,1,1)$-quasi-product set with $\#\T_\rho\sim \rho^{-1}$ such that
		$$\#\cP_{r'}(\T_\rho)\approx \frac{\rho^{-2}}{(r')^\alpha}.$$
		In the previous remark we have proved such a choice is possible. Pack each $\tau\in\T_\rho$ with $K_1K_2\frac\rho\delta$ many $\delta$-tubes, call them $\T_\tau$, as follows. There are $\rho/\delta\gtrsim K_1(\rho/\delta)^{s}$ possible directions. We arbitrarily select $K_1(\rho/\delta)^{s}=K_1K_2$ many of them. For each selected direction we pick all $\delta$-tubes $T$ in that direction, lying inside  $\tau$. There are $\rho/\delta=K_2(\rho/\delta)^{1-s}$ many such $T$. Now let $\T=\cup_{\tau\in\T_\rho}\T_\tau$. Note that $\T$ is a $(\delta,s,1-s,K_1,K_2)$-quasi-product set $\#\T\sim K_1K_2\delta^{-1}$.
		
		Note that each $\delta$-square inside each $\rho$-square in  
		$\cP_{r'}(\T_\rho)$ is in $\cP_{r}(\T)$. Thus
		$$ \# \cP_{r}(\T)\approx (\frac\rho\delta)^2\#\cP_{r'}(\T_\rho)\approx (K_1K_2)^\alpha \frac{\delta^{-2}}{r^\alpha}\sim (K_1K_2)^{\alpha-1}\frac{\delta^{-1}\#\T}{r^\alpha}.$$
	\end{remark}	
	
	\medskip
	
	Given $\T$ and a shading $Y$ of $\T$ we write
	$$\cP_r(\T,Y)=\{p:\;\#\{T\in\T:\;p\in Y(T)\}\sim r\}.$$
	
	\begin{cor}
		\label{nononecase}
		Assume $s+d\not=1$. Let $\sigma=\min\{s+d,2-s-d\}$. Then for each $(\delta,s,d,K_1,K_2)$-quasi-product set $\T$, each shading $Y$ of $\T$ such that $Y(T)$ is a $(\delta,\sigma,K_3)$-set,  and each $r\ge 1$ we have 
		$$\#\cP_r(\T,Y)\les K_3^{\frac13(2+\frac{d}{s})}(K_1K_2)^{1+\frac{d}s} \frac{\delta^{-s-d}\#\T}{r^{2+\frac{d}{s}}}$$
		if $s>d$ and 
		$$\#\cP_r(\T,Y)\les K_3(K_1K_2)^2 \frac{\delta^{-s-d}\#\T}{r^{3}}$$
		if $s\le d$.
	\end{cor}
	\begin{proof}
		We apply Theorem \ref{maininct}  to $\T$ and the refined shading $Y'(T)=\{p\in\cP_r(\T,Y)\cap Y(T)\}$. We have $$\#\cP_r(\T,Y)r\sim \sum_{T\in\T}\#Y'(T)\les  K_3^{\frac13}(K_1K_2)^{1-\frac1\alpha} (\delta^{-s-d}\#\T)^{1/{\alpha}}\#\cP_r(\T,Y)^{1-\frac1\alpha},$$
		or
		$$\#\cP_r(\T,Y)\les K_3^{\frac\alpha3}(K_1K_2)^{\alpha-1}\frac{\delta^{-s-d}\#\T}{r^\alpha}.$$	
		Since $\alpha=\min\{3,2+\frac{d}{s}\}$, the result follows.
	\end{proof}
	\begin{remark}
		\label{rwells}Our method offers a different perspective on  inequalities proved in \cite{guthincidence} via the high-low method.
		More precisely, our argument can be easily adapted to prove the inequality
		$$\#\cP_r(\T)\les \frac{(\#\T)^2}{r^3},\;\;r\ge 1$$
		for well-spaced collections $\T$ consisting of $\sim \delta^{-1}$ many  $\delta^{1/2}$-separated $\delta$-tubes $T$. This is because the well-spaced property persists at smaller scales, as demanded by Proposition \ref{fkjgu8rtu9i6iy-}.

		As a  consequence of this, the same inequality will hold for collections $\T_W$ consisting of $\sim W^2$ many $W^{-1}$-separated $\delta$-tubes, if $1\le W\le \delta^{-1/2}$. Indeed, each tube  $T\in \T_W$ lies inside a unique $W^{-1}$-tube $T_W$. No two tubes $T$ lie inside the same $T_W$, so the collection $\T$ of these fat tubes has cardinality $\sim W^2$. Moreover,
		$$\#\cP_r(\T_W)\le \#\cP_r(\T)\les \frac{(\#\T)^2}{r^3}\sim \frac{(\#\T_W)^2}{r^3}.$$   	
		
	\end{remark}
	\section{Fourier decay of fractal measures}
	\label{sec:4}
	This section investigates three rather different approaches to Conjecture \ref{hfrhvouuigrtiogupioythu09}. We begin by describing the state of the art prior to our work.
	\subsection{The high-low method}
	\label{ss4.1}
	Conjecture \ref{hfrhvouuigrtiogupioythu09} was proved in the range $s\ge \frac23$, first in \cite{O2} for the parabola, then in \cite{Yi} for arbitrary curves. 
	
	The same method can be used to prove the estimate \eqref{kjrefut0=or0i569y0-6iu09} (that is not sharp) in the range $s< \frac23$. We begin by describing this approach. We only sketch a simplified high-level proof of the results in \cite{O2}, \cite{Yi}, that lines up with our forthcoming arguments. Consider the (higher order) additive $\delta$-energy of
	a finite set $S\subset\R^2$
	$$\E_{3,\delta}(S)=\#\{(s_1,\ldots,s_6)\in S^6:\;|s_1+s_2+s_3-s_4-s_5-s_6|\lesssim \delta\}.$$
	We use the following simpler variant of Proposition \ref{p o4i9i59ui=0g6-3}, that follows via an essentially identical argument.
	\begin{proposition}
		\label{pl rj iuguh8u-49}
		Let $\Gamma$ be a smooth curve.  Fix $\delta>0$, $s\in [0,1]$ and $A\gtrsim 1$.
		Assume the uniform estimate
		$$
		\E_{3,\delta}(S)\le A\delta^{-3s}
		$$
		holds for each $(\delta,s)$-KT set  $S\subset \Gamma$.
		Let $R=\delta^{-1}$. Then the estimate
		$$\int_{B_R}|\widehat{\mu}|^6\les AR^{2-3s}$$
		holds for each $\mu$ supported on $\Gamma$, that satisfies \eqref{e4 giobjytibuiuiu}.
	\end{proposition}
	
	The method in \cite{O2}, \cite{Yi} relies on the following incidence estimate. An earlier version of this estimate for neighborhoods of lines was proved in \cite{furen} (Theorem 5.2) using the high-low method. 
	\begin{theorem}[\cite{O2}, Theorem 6.13; \cite{Yi}, Theorem 5.1]
		\label{tOFR}
		Let $\cN_\de(\Gamma)$ be the $\delta$-neighborhood of $\Gamma$.
		Let $u\in[0,1]$	 and $t\in[0,2]$ be such that $t+u\le 2$, and let $C_1,C_2\ge 1$. Let $\cP\subset \cD_\delta$ be a $(\delta,t,C_1)$-KT set.
		
		For each $p\in\cP$ let $\cF(p)\subset\{q\in\cD_\delta:\;q\cap (p+\cN_\de(\Gamma))\not=\emptyset\}$ be a $(\delta,u,C_2)$-KT set. Then, writing $\cF=\cup_{p\in\cP}\cF(p)$ we have 
		$$\sum_{p\in\cP}\#\cF(p)\les \sqrt{\delta^{-1}C_1C_2\#\cF\#\cP}.$$ 	
	\end{theorem}
	Let us first see the proof of \eqref{f[poi 0h0oi0-j9i9]} in the range $s\ge \frac23$.
	Consider a set $S$ as in Proposition \ref{pl rj iuguh8u-49}. 
	This is a $(\delta,s,1)$-KT set, in particular also a $(\delta,\frac23,\delta^{\frac23-s})$-KT set.
	
	Let $\cP$ be a minimal collection of squares in  $\cD_\delta$ containing the points $S+S$. We assume $\cP$ is a $(\delta,\frac43,\delta^{\frac43-2s})$-KT set. This is  easily checked to be the case if instead of $S+S$ we consider $S_1+S_2$, where $S_1,S_2$ are  subsets of $S$ separated by $\sim 1$. It is this part of the argument that needs $|\gamma''|\gtrsim 1$, since this guarantees $F(x,y)=(x+y,\gamma(x)+\gamma(y))$ is bi-Lipschitz on $S_1\times S_2$. Executing this argument rigorously  would require a bilinear to linear reduction similar to the one presented in the next subsection.

	Note that 
	$$\E_{3,\delta}(S)\sim \sum_{s\in \cP+S}(\#\{p\in\cP:\;s\cap (p+S)\not=\emptyset\})^2.$$
	Write $\cQ_r$ for the collection of those squares in $\cP+S$ that intersect $\sim r$ many sets $p+S$. Then the above formula shows that
	\begin{equation}
		\label{rjeu8ugihoiyth9iy0-}
		\E_{3,\delta}(S)\sim \sum_{r=1
			\atop{dyadic}}^{\#\cP}\#\cQ_rr^2.
	\end{equation}
	For each $p\in\cP$ define $\cF(p)=(p+S)\cap \cQ_r\subset p+\cN_\de(\Gamma)$. Since $p+S$ is a $(\delta,\frac23,\delta^{\frac23-s})$-KT set, so is $\cF(p)$. Note that 
	$$\cF=\bigcup_{p\in\cP}\cF(p)=\cQ_r.$$
	We apply Theorem \ref{tOFR} with $u=\frac23$, $t=\frac43$, $C_2=\delta^{\frac23-s}$, $C_1=\delta^{\frac43-2s}$ to get 
	$$r\#\cQ_r\sim \sum_{p\in\cP}\#\cF(p)\les \sqrt{\delta^{-1}\delta^{2-3s}\#\cQ_r\#\cP}.$$ 
	This gives 
	$$\#\cQ_r\les\delta^{1-5s}r^{-2}.$$Plugging this back into \eqref{rjeu8ugihoiyth9iy0-} leads to the bound
	$$\E_{3,\delta}(S)\les (\delta^{-1})^{5s-1}.$$
	This in turn implies  \eqref{f[poi 0h0oi0-j9i9]} for $s\ge \frac23$ via Proposition \ref{pl rj iuguh8u-49}.
	\\
	\\
	When $s\le \frac23$ we may apply Theorem \ref{tOFR} with $u=s$, $t=2s$, $C_1,C_2=1$ to get 
	$$r\#\cQ_r\sim \sum_{p\in\cP}\#\cF(p)\les \sqrt{\delta^{-1}\#\cQ_r\#\cP},$$
	or
	$$\#\cQ_r\les \delta^{-1-2s}.$$
	Plugging this into \eqref{rjeu8ugihoiyth9iy0-} leads to the bound
	\begin{equation}
		\label{huifhuryfy7y7gy7gy7}	
		\E_{3,\delta}(S)\les(\delta^{-1})^{1+2s}=(\delta^{-1})^{5s-1+\beta},\;\;\;\beta=2-3s.
	\end{equation}
	Using again Proposition \ref{pl rj iuguh8u-49} leads to the estimate for $s\le \frac23$
	\begin{equation}
		\label{kjrefut0=or0i569y0-6iu09}
		\int_{B_R}|\widehat{\mu}|^6\les R^{3-4s}.
	\end{equation}
	This estimate loses its strength as $s$ approaches $\frac12$. For example, at $s=\frac12$ it gives the upper bound $R$, while the result in \cite{demwangszem} proves the upper bound $R^{7/8}$.

	\subsection{Energy estimates for the parabola}
	We prove Theorem \ref{io jfpurei90 h9=hhytj uy uythytj  }, that we recall below.
	\begin{theorem}
		\label{io jfpurei90 h9=h}
		Assume $0<s\le \frac23$. Assume $\mu$ is supported on $\PP^1$ and satisfies  the Frostman condition
		\begin{equation}
			\label{e4}
			\mu(B(y,r))\lesssim r^s
		\end{equation}
		for each $y\in \R^2$ and each $r>0$. Then we have
		\begin{equation}
			\label{cjfgutgutugtuhu-h9u}
			\int_{B_R}|\widehat{\mu}|^6\les R^{2-\frac{5s}2}\|\mu\|.
		\end{equation}
	\end{theorem}
	\begin{remark}
		When $s<\frac23$, the exponent in \eqref{cjfgutgutugtuhu-h9u} is smaller than the one in \eqref{kjrefut0=or0i569y0-6iu09}. The two exponents are equal when $s=2/3$.
	\end{remark}
	We will derive \eqref{cjfgutgutugtuhu-h9u} as a consequence of an estimate for the trilinear $\delta$-energy.
	Given finite sets $S_1,S_2,S_3\subset \R^2$ we define 
	$$\E_{3,\delta}(S_1,S_2,S_3)=\#\{(s_1,\ldots,t_3)\in (S_1\times S_2\times S_3)^2 :\;|s_1+s_2+s_3-t_1-t_2-t_3|=O(\delta)\}.$$
	\begin{theorem}
		\label{jsdhfuurt gi  topgi yhi}
		Let $0<s\le \frac23$.
		Consider three $(\delta,s)$-KT subsets  $S_1,S_2,S_3$ of $\PP^1$ lying above intervals $J_i\subset [-1,1]$. We assume $\dist(J_1,J_2)$, $\dist(J_2,J_3)$, $\dist(J_3,J_1)\sim 1$. Then
		$$\E_{3,\delta}(S_1,S_2,S_3)\les\delta^{-s}(\#S_1\#S_2\#S_3)^{5/6}.$$
	\end{theorem}
	\smallskip
	
	\begin{remark}
		\label{chcurt ug8u-g06u9=67u0978}
		Theorem \ref{jsdhfuurt gi  topgi yhi} remains true if the separation $\sim 1$ is replaced with $\sim 1/K$, for some $K\ge 1$. The implicit constant in $\les$ will now also depend on $K$.
	\end{remark}
	
	We first show how  to interpret the energy bound as an incidence estimate \eqref{fki9hik09hi-h0}, and then use Theorem \ref{huufwf[gmiu9i96=y]} to prove it. 
	Given any $(t_1,t_2,t_3)\in S_1\times S_2\times S_3$ we write
	$$N(t_1,t_2,t_3)=\#\{(s_1,s_2,s_3)\in S_1\times S_2\times S_3:\;|s_1+s_2+s_3-t_1-t_2-t_3|=O(\delta)\}.$$
	Note that
	$$\E_{3,\delta}(S_1,S_2,S_3)= \sum_{(t_1,t_2,t_3)\in S_1\times S_2\times S_3}N(t_1,t_2,t_3).$$
	Theorem \ref{jsdhfuurt gi  topgi yhi} will follow once we prove 
	that 
	\begin{equation}
		\label{kjfjigioguitouiortu}
		\sum_{(t_1,t_2,t_3)\in S_1\times S_2\times S_3}N(t_1,t_2,t_3)\les \delta^{-s}(\#S_1\#S_2\#S_3)^{5/6}.
	\end{equation}
	Let us now recall the framework from \cite{O1}, with slight changes in our notation.
	All arguments will work for permutations of the indices $1,2,3$.
	
	Consider the collection $\SSS$ of squares in $\cD_\delta$ that cover the points 
	$$\{(3(x_1+x_2),\sqrt{3}(x_1-x_2)):\;(x_1,x_1^2)\in S_1,\;(x_2,x_2^2)\in S_2\}.$$
	Each $(t_1,t_2,t_3)\in S_1\times S_2\times S_3$ gives rise to the circle $C_{t_1,t_2,t_3}$ centered at $(2\sigma_1,0)$ with radius $\sqrt{6\sigma_2-2\sigma_1^2}$, where $(\sigma_1,\sigma_2)=t_1+t_2+t_3.$ Different triples may give rise to the same circle.

	We call $S_{t_1,t_2,t_3}(\delta)$ the $O(\delta)$-neighborhood of this circle.
	We identify $C_{t_1,t_2,t_3}(\delta)$ with $C_{t_1',t_2',t_3'}(\delta)$ if 
	$$|t_1'+t_2'+t_3'-t_1-t_2-t_3|=O(\delta).$$
	We denote by $\C$ the family of distinct $C_{t_1,t_2,t_3}(\delta)$.

	If $s_i=(x_i,x_i^2)$  then 
	\begin{equation}
		\label{edjrfug9u8-689y8650}
		|s_1+s_2+s_3-t_1-t_2-t_3|=O(\delta)
	\end{equation}
	implies that the $\delta$-square  $q\in\SSS$ containing $(3(x_1+x_2),\sqrt{3}(x_1-x_2))$ is incident to $C_{t_1,t_2,t_3}(\delta)$. 
	It follows that 
	$$\sum_{(t_1,t_2,t_3)\in S_1\times S_2\times S_3}N(t_1,t_2,t_3)\sim \sum_{C\in \C}\;(\#\{q\in\SSS:\;q\cap C\not=\emptyset\})^2.$$

	Due to symmetry, we may assume that $J_2$ lies at the left of $J_1$. Let $Y=\dist(J_1,J_2)\sqrt{3}$. Thus
	all squares in $\SSS$ lie inside $\{(x,y):\;y>Y\}$. We may also restrict attention to the arc of the circle $C_{t_1,t_2,t_3}^{+}=C_{t_1,t_2,t_3}\cap \{(x,y):\;y>Y\}$ and its $O(\delta)$-neighborhood $C_{t_1,t_2,t_3}^{+}(\delta)$.  
	Next, as shown in \cite{O1}, there is a map $$\cG:\{(x,y):\;y>0\}\to \D=\{(x,y):\;x^2+y^2<1\}$$ that sends each $C_{t_1,t_2,t_3}\cap\{(x,y):\;y>0\}$ to an open  line segment $l(\sigma_1,\sigma_2)$ connecting two points on $\cS^1$.
	While this is not relevant for our discussion, the endpoints of this segment have coordinates $(\frac{\xi^2-1}{\xi^2+1},-\frac{2\xi}{\xi^2+1})$, where $\xi=2\sigma_1\pm \sqrt{6\sigma_2-2\sigma_1^2}$ are the $x$-coordinates of the  two intersections of $C_{t_1,t_2,t_3}$ with the $x$-axis. We will consider the restriction of $\cG$
	$$\cG':\{(x,y):\;y>Y\}\to \D'=\{(x,y):\;x^2+y^2<\eta\},$$
	where  $\eta<1$. This map is bi-Lipschitz with constants $\sim 1$.
	
	Note that writing $t_i=(y_i,y_i^2)$ 
	we have$$6\sigma_2-2\sigma_1^2=2(y_1-y_2)^2+2(y_2-y_3)^2+2(y_3-y_1)^2\sim 1.$$
	Let us call $L(\sigma_1,\sigma_2)\in \cA(2,1)$ the line  which $l(\sigma_1,\sigma_2)$ is a part of. 
	
	Lemma 4.39 and Lemma 4.44 in \cite
	{O1} prove that the map
	$$(\sigma_1,\sigma_2)\mapsto L(\sigma_1,\sigma_2)$$
	is bi-Lipschitz on $\{(\sigma_1,\sigma_2):\;6\sigma_2-2\sigma_1^2\sim 1\}$.

	Since $\cG'$ is bi-Lipschitz, it  maps each $C_{t_1,t_2,t_3}^+(\delta)$ to an approximate $\delta$-tube, and each $\delta$-square in $\SSS$ to another (approximate) $\delta$-square. We call $\T$ this collection of tubes, and we call $\PP$ this collection of squares. It follows that
	\begin{equation}
		\label{fki9hik09hi-h0}
		\E_{3,\delta}(S_1,S_2,S_3)\sim \sum_{T\in\T}(\#\{p\in\PP:\;p\cap T\not=\emptyset\})^2.
	\end{equation}
	Note that both $\T$ and $\PP$ live inside the disk $\D'$.

	The collection $Y'(p)$ of all tubes incident to a square  $p\in \PP$ is a $(\delta,s)$-KT set with size
	\begin{equation}
		\label{ juryegu rtugy t8ug8tug-9i90}
		\#Y'(p)\lesssim\#S_3.
	\end{equation}
	Indeed, for any given $(3(x_1+x_2),\sqrt{3}(x_1-x_2))$, the incident circles $C_{t_1,t_2,t_3}$ are centered at $(2\sigma_1,0)$ and have radii $\sqrt{6\sigma_2-2\sigma_1^2}$, where $$(\sigma_1,\sigma_2)=(x_1,x_1^2)+(x_2,x_2^2)+(x_3,x_3^2)+O(\delta)$$ for some $(x_3,x_3^2)\in S_3$. The corresponding set of $(\sigma_1,\sigma_2)$ is thus a $(\delta,s)$-KT set, and we recall that the map $L$ is bi-Lipschitz.  
	\medskip

	The pair $(\PP,Y')$ satisfies two of the three requirements of Theorem \ref{tWW25}, with $t=2s$, in the dual form.
	First, $\PP$ can be seen to be a $(\delta,2s)$-KT set, with size
	\begin{equation}
		\label{ juryegu rtugy t8ug8tug-9i902}
		\#\PP\lesssim\#S_1\#S_2.
	\end{equation}
	Recall that $\sigma = \min \{2s,2-2s\}.$ each $Y'(p)$ is a $(\delta,\sigma)$-KT set since it is a $(\delta,s)$-KT set, and $s\le \sigma$ when $s\le \frac23.$ However, there is no guarantee that $Y'(s)$ is two-ends. Instead, we will show in Lemma \ref{2endsredu} that most of $Y'(p)$ is concentrated inside an arc of $\cS^1$, whose rescaled copy is two-ends. Because of the smallness of this arc,  Theorem \ref{tWW25} will need to be applied on subsets of squares and tubes localized inside certain rectangles in $[0,1]^2$. Due to rescaling, Theorem \ref{tWW25} will demand a stronger non-concentration property of $\PP$, that we call {\em rectangular} $(\delta,s)$-KT. We verify this property in Proposition \ref{k jg  tgu hu-uh49-t09359056 h9876} and
	Corollary \ref{idcj uhuireyygiytigy}, using in a crucial way the Cartesian product flavor of the set $\PP$.
	
	The following result will be applied in its dual form. We formulate it this way for extra clarity. 
	\begin{lemma}
		\label{2endsredu}	
		Let $T$ be a $\delta$-tube with a shading $Y(T)$ that is a $(\delta,s)$-KT set with cardinality $P\le \delta^{-s}$. Let $\epsilon>0$. 
		
		There is a segment $\tau$ of $T$ of length $L$ at least $(\delta^sP)^{\frac1{s-\epsilon^2}}$, that contains  $N\ge L^{\epsilon^2}P$ of the squares in $Y(T)$, call them $Y(\tau)$, and such that 
		\begin{equation}
			\label{pojfiruegioutg}
			\#(Y(\tau)\cap B(x,L(\delta/L)^{\epsilon}))\le (\delta/L)^{\epsilon^3}N,\;\text{for each }x.
		\end{equation}
	\end{lemma}
	\begin{proof}
		We test \eqref{pojfiruegioutg} with $\tau=T$. If it is true, we stop. Otherwise we find a segment $\tau_1$ with length $L_1=\delta^{\epsilon}$ that contains a number $N_1$ of squares in $Y(T)$ satisfying
		$$N_1>\delta^{\epsilon^3}P.$$
		We next test whether $\tau_1$ satisfies \eqref{pojfiruegioutg}. If it does, we stop. Otherwise, we may find a segment $\tau_2$ of $\tau_1$ with length $L_2=L_1(\delta/L_1)^\epsilon$ that contains $N_2$ squares from $Y(T)$ satisfying
		$$N_2>(\delta/L_1)^{\epsilon^3}N_1.$$ 
		Assuming this process will continue for $n$ steps, we find a segment $\tau_n$ of $\tau_{n-1}$ with length $L_n=L_{n-1}(\delta/L_{n-1})^\epsilon$, containing $N_n$ squares such that
		$$N_n>(\delta/L_{n-1})^{\epsilon^3}N_{n-1}.$$
		This process must stop at some point, with the values $L_n$, $N_n$ . Note that $L_n=\delta^{1-(1-\epsilon)^n}$, and thus
		$$N_n>\delta^{\epsilon^3}(\delta/L_1)^{\epsilon^3}\cdot\ldots\cdot(\delta/L_{n-1})^{\epsilon^3}=\delta^{\epsilon^2(1-(1-\epsilon)^n)}P=L_n^{\epsilon^2}P.$$
		Since $Y(T)$ is a $(\delta,s)$-KT set, we must have
		$$N_n\le (L_n/\delta)^{s}.$$
		It follows that the terminal $L_n$ cannot be smaller than $(\delta^s P)^{\frac1{s-\epsilon^2}}$. We may take $L=L_n$ and $N=N_n$.
		
	\end{proof}
	
	Recall that we write $\cN_r(S)$ for the $r$-neighborhood of $S$.

	\begin{proposition}
		\label{k jg  tgu hu-uh49-t09359056 h9876}
		Let $\delta<r'\le r\lesssim 1$.
		
		Let $\Gamma=\{(t,\gamma(t)),\;t\in[a,b]\}$ be  a smooth curve with length $r$. We assume $\|\gamma'\|_\infty\lesssim 1$.
		
		Let $S\subset \R$ be a $(\delta,s)$-KT set. Then 
		$$\#((S\times S)\cap \cN_{O(r')}(\Gamma))\lesssim (\frac{rr'}{\delta^2})^s.$$ 
		The implicit constant is independent of $\delta,r,r'$ and $\gamma$.
	\end{proposition}
	\begin{proof}

		We cover $(S\times S)\cap \cN_{O(r')}(\Gamma)$ with pairwise disjoint  squares $q=I_q\times J_q$ with side length $O(r')$. Then
		$$\#((S\times S)\cap \cN_{O(r')}(\Gamma))=\sum_q\#(S\cap I_q)\#(S\cap J_q)\lesssim (\frac{r'}{\delta})^s\sum_q\#(S\cap I_q).$$
		Since $\gamma$ has bounded first derivative, an interval $I$ may coincide with $I_q$ for at most  $O(1)$ many squares $q$.  Since $\Gamma$ has length $r$, all these $I$ lie inside an interval of length $\lesssim r$. Using that $S$ is $(\delta,s)$-KT we find that 
		$$\sum_q\#(S\cap I_q)\lesssim (\frac{r}{\delta})^s.$$
		
	\end{proof}	
	Recall the definition of a rectangular KT-set in Definition \ref{def.rect}.
	\begin{corollary}
		\label{idcj uhuireyygiytigy}
		The set $\PP$ is a rectangular $(\delta,2s)$-KT set.
	\end{corollary}
	\begin{proof}
		Recall the map $\cG'$ and call
		$\cA(x_1,x_2)=(3(x_1+x_2),\sqrt{3}(x_1-x_2))$. Consider a box $B\subset \D'$ with dimensions $(r',r)$, $\delta\le r'\le r$ and central long axis given by a line segment $l$ of length $r$. Then $(\cG')^{-1}(l)$ is part of a  circle centered on the $x$-axis, lying inside the $y>Y$ half-space. Since $\cG'\circ\cA$ is bi-Lipschitz on $J_1\times J_2$,  we may cover 
		$(\cG'\circ \cA)^{-1}(l)$ with $O(1)$ many  $\Gamma$ with length $\lesssim r$. Note that
		$$\#(\PP\cap B)\sim \sum_{\Gamma}\#((S_1\times S_2)\cap \cN_{O(r')}(\Gamma)).$$
		The result follows from Proposition \ref{k jg  tgu hu-uh49-t09359056 h9876} applied to $S=S_1\cup S_2$.
		
	\end{proof}
	
	We now prove the most substantial result of this section, Theorem \ref{thm.rectangles1}, which may be of independent interest. We recall it below for convenience. 
	\begin{theorem}
		\label{huufwf[gmiu9i96=y]}\label{thm.rectangles}
		Assume $0\le s\le 1$.	
		Consider a collection $\T$ of $\delta$-tubes and a collection $\PP$ of $\delta$-squares in $[0,1]^2$.
		
		Assume that  $\PP$ is a rectangular $(\delta,2s)$-KT set for each $p\in\PP$. Assume that $Y'(p)=\{T\in\T:\;p\cap T\not=\emptyset\}$ is a  $(\delta,\sigma)$-KT set, $\sigma=\min(2s,2-2s)$.
		Write
		$$I(\T,\PP)=\sum_{p\in\PP}\#Y'(p).$$
		Then 
		\begin{equation}
			\label{ orjifu54it0-96-py=56-56}
			I(\T,\PP)\les \delta^{-\frac{2s}{3}}(\#\T)^{2/3}(\#\PP)^{1/3}.
		\end{equation}
	\end{theorem}
	\begin{proof}
		Fix $\epsilon>0$.
		There will be various losses of the form $(1/\delta)^{O(\epsilon)}$ that will be hidden in the notation $\les$ and $\approx$. We use pigeonholing to enforce various  uniformity assumptions. These will only cost losses of order $\les 1$. 
		\\
		\\
		Step 1. (Parameters $N,L$) We restrict attention to all $p$ in some subset  $\PP_1=\PP$, that have the  parameter $\#Y'(p)$ in the same dyadic range, $\#Y'(p)\sim P$. 
		We choose $P$ so that the new number of incidences remains roughly the same as the old one  $$\#I(\T,\PP_1)\approx \#I(\T,\PP).$$
		
		We apply Lemma \ref{2endsredu} to each such  $p$ to produce an arc $\theta_p\subset\cS^1$ of length $L_p\ge \delta$ such that the collection - denoted by $\T(\theta_p)$ - of those tubes  $T\in Y'(p)$ with directions $d(T)$ inside $\theta_p$ is $(\epsilon,\epsilon^3)$-two-ends. What this means is that these directions, after being rescaled by $1/L_p$, form an $(\epsilon,\epsilon^3)$-two-ends subset of $\cS^1$ at scale  $\delta/L_p$.
		
		We may also assume that $L_p\sim L$, $\#Y'(\theta_p)\sim N$ for all $p\in \PP_1$, for some fixed dyadic parameters $L$ and $N\approx P$.
		The parameter $P$ will not be mentioned again, as it is replaced with $N$. Note that
		\begin{equation}
			\label{ jutgugtg9 yh =390h yioh0yi0h-}
			\#I(\T,\PP)\approx N\#\PP_1.
		\end{equation}
		We rename $Y'(p)$ to refer to the smaller collection $Y'(\theta_p)$, so 
		\begin{equation}
			\label{;lckjfvjvjiobiytpio}
			\#Y'(p)\sim N,\;\;\forall p\in\PP_1.
		\end{equation}
		\\
		\\    
		Step 2. (Rectangles $R$, segments $\gamma$)
		We fix a maximal $L$-separated set of directions $D_{L}\subset \cS^1$.
		For each direction in $ d \in D_{L}$, we tile/cover $[0,1]^2$ with $\sim 1/L$ many $(L,1)$-rectangles $R,$ with the long side pointing in direction $d.$ There are $\sim L^{-2}$ many such $R$, call them $\cR$. It is worth noting that the eccentricity of $R$ coincides with the length of an arc $\theta_p$. 
		
		Each  $T\in \T$ fits inside a unique $R$. We call  $\T_R$ the collection of these $T$, and note that 
		\begin{equation}
			\label{cjuuig8 rtui90 yti h9yi 0968}
			\sum_{R\in\cR}\#\T_R=\#\T.
		\end{equation}
		Simple geometry shows that if $T,T'\in\cT_R$ have nonempty intersection, it must be (roughly speaking) a segment with dimensions $(\delta,k\delta/L)$, for some $k\in\N$. It thus makes sense to tile $R$ with a collection $\Gamma_{R,all}$ of segments $\gamma$ with dimensions $(\delta,\delta/L)$, having the same orientation and eccentricity as $R$. Each $T\in\T$ may be thought of as the union of such $\gamma$.
		\\
		\\
		Step 3. (The squares $\PP_R$) For $R\in\cR$, we call $\PP_R$ the collection of those $p\in \PP_1$ that lie inside $R$, and whose arc $\theta_p$ contains the direction $d(R)$  of $R$.
		
		Thus, the sets $(\PP_R)_{R\in\cR}$ are pairwise disjoint, which implies that
		\begin{equation}
			\label{cjuuig8 rtui90 yti h9yi 09682}
			\sum_{R\in\cR}\#\PP_R=\#\PP_1.
		\end{equation}
		We note that for each $p\in\PP_R$,   $Y'(p)$ coincides with those $T\in\T_R$ such that $p\cap T\not=\emptyset$.
		This is because the eccentricity and orientation of $R$ match the length $L$ and location of the arc $\theta_p$. This will allow us to perform the incidence count separately for each $R$.
		\\
		\\
		Step 4. (Parameter $M$)
		Each $p\in \PP_R$ lies inside a unique $\gamma\subset R$. It follows that 
		$$\#I(\T,\PP)\approx \sum_{R}\sum_{M}\sum_{\gamma\in\Gamma_{ R,all}\atop{\#(\PP_R\cap \gamma)\sim M}}\sum_{p\in\PP_R\cap \gamma}\#Y'(p).$$
		We pick $M$ such that
		$$\#I(\T,\PP)\approx \sum_{R}\sum_{\gamma\in \Gamma_R}\sum_{p\in\PP_R\cap \gamma}\#Y'(p).$$
		where $$\Gamma_R=\{\gamma\in \Gamma_ {R,all}:\;\#(\PP_R\cap \gamma)\sim M\}.$$

		Write
		$\Gamma=\cup_R\Gamma_R.$ Note that due to \eqref{;lckjfvjvjiobiytpio}, \eqref{cjuuig8 rtui90 yti h9yi 09682} and the choice of $M$ we have
		\begin{equation}
			\label{cjuuig8 rtui90 yti h9yi 09683}
			\#\Gamma\approx \frac{\#\PP_1}{M}.
		\end{equation}
		\\
		\\
		Step 5. (Contracting the incidence graph)
		An earlier observation implies that for  each $p\in\PP_R\cap \gamma$ with $\gamma\in\Gamma_R$
		$$Y'(p)=\{T\in\T_R:\;\gamma\subset T\}.$$
		This shows that $Y'(p)$ takes the same value when $p\in\PP_R\cap \gamma$. Thus, we may recast the incidence graph between $\T$ and $\PP_1$ as a new graph  with vertex set $\T\sqcup\Gamma$, and an edge between $\gamma$ and $T$ if and only if $\gamma\subset T$. Note that
		$$\#I(\T,\PP_1)\approx M^{-1}\#I(\T,\Gamma).$$ 
		\\
		\\
		Step 6. (Rescaling $R$) Let $\bar{\delta}=\frac{\delta}{L}$. Fix  $R$. 
		We rescale it by $(\frac{1}{L},1)$ to get $[0,1]^2$. Then each $T\in\T_R$ becomes a $\bar{\delta}$-tube, and we call this collection $\bar{\T}_R$. Also, each $\gamma\in\Gamma_R$ becomes a $\bar{\delta}$-square, and we call this collection $\bar{\PP}_R$. Then $\#\bar{\PP}_R=\#\Gamma_R$ and $\#\bar{\T}_R=\#\T_R$. We note that \ $\bar{\delta}\le 1$. Thus, using \eqref{cjuuig8 rtui90 yti h9yi 0968} and \eqref{cjuuig8 rtui90 yti h9yi 09683} we find
		\begin{equation}
			\label{jtgi5iyi670uoi97i0-67iu0-i}
			\#\PP_1\sum_{R\in\cR}\#\bar{\T}_R\approx M\#\T\sum_{R\in\cR}\#\bar{\PP}_R.
		\end{equation}
		It will be important to realize that our rescaling magnifies angles between tubes by a factor of $1/L$.
		\\
		\\
		Step 7. (Preservation of spacing properties  via rescaling) The following properties can be easily seen to be true for each $R$.
		
		First, $\bar{\PP}_R$ is a $(\bar{\delta},2s,K_1)$-KT set with $K_1\sim \frac{1}{ML^{s}}$. Indeed, after rescaling, the number of $\bar{p}\in\bar{\PP}_R$ lying inside an $r$-ball with $r\ge \bar{\delta}$ is dominated by $1/M$ times the number of $p\in\PP$ that lie inside a rectangle with dimensions $(rL,r)$. Since $\PP$ was assumed to be rectangular $(\delta,2s)$-KT, this number is at most $$\frac1M(\frac{r\sqrt{L}}{\delta})^{2s}=\frac{1}{ML^s}(\frac{r}{\bar{\delta}})^{2s}.$$

		Second, since each $Y'(p)$ was assumed to be $ (\delta,\sigma)$-KT,  for each $\bar{p}\in\bar{\PP}_R$ the (dual) shading $\bar{Y}'(\bar{p})=\{\bar{T}\in\bar{\T}_R:\;\bar{p}\cap \bar{T}\not=\emptyset\}$ is  $(\bar{\delta},\sigma,K_2)$-KT set, with $K_2\sim 1$. Moreover,
		$\#\bar{Y}'(\bar{p})\sim N$.
		\\
		\\
		Step 8. (Two-ends preservation via rescaling) We prove that $\bar{Y}'(\bar{p})$ is $(\epsilon,\epsilon^3)$-two-ends with respect to the scale $\bar{\delta}$.
		Call $\gamma$ the square that gets mapped to $\bar{p}$. Recall that $\theta_p=\theta_{p'}$ and  $Y'(p)=Y'(p')$ when $p,p'\in \PP_R\cap \gamma$. Call $\theta$ the corresponding arc $\theta_p$.
		Call $\Theta$ the rescaled arc, with length $\sim 1$.
		
		We consider an arbitrary segment of $\Theta$ of length $\bar{\delta}^{\epsilon}$ and prove that it contains - call this number $U$ - at most a $\bar{\delta}^{\epsilon^3}$-fraction of $\bar{Y}'(\bar{p})$. Via rescaling, this number coincides with the number - call it $V$ - of $T\in\T$ with directions inside some segment of $\theta'\subset \theta$
		of length $\sim \bar{\delta}^{\epsilon}L=(\delta/L)^\epsilon L$. Our construction of $Y'(\theta)$ is aligned with \eqref{pojfiruegioutg}, thus 
		$$V \le (\delta/L)^{\epsilon^3}\#Y'(\theta)\sim(\delta/L)^{\epsilon^3}N= \bar{\delta}^{\epsilon^3}N.$$ 
		This implies the desired estimate 
		$$U=V\le \bar{\delta}^{\epsilon^3}N\sim  \bar{\delta}^{\epsilon^3}\#\bar{Y}'(\bar{p}). $$	
		
		Step 9.  We apply Theorem \ref{completeWW} to the pair $(\bar{\PP}_R,\bar{Y}')$, using the point-line duality
		$$\#\bar{\PP}_RN^{3/2}\sim\sum_{\bar{p}\in\bar{\PP}_R}\#\bar{Y}'(\bar{p})N^{1/2}\les \bar{\delta}^{-s}\#\bar{\T}_RK_1K_2^{1/2}\sim \bar{\delta}^{-s}\#\bar{\T}_R\frac{1}{ML^{s}}.$$	
		When combined with	\eqref{jtgi5iyi670uoi97i0-67iu0-i} and summation in $R$ we get
		\begin{equation}
			\label{jfrtgupgi-46iy9i4-u}
			\frac{\#\PP_1}{\#\T}N^{3/2}\les (\frac{L}{\delta})^sL^{-s}=\delta^{-s}.
		\end{equation}
		Finally, recalling \eqref{ jutgugtg9 yh =390h yioh0yi0h-} we get the desired inequality
		$$I(\T,\PP)\approx N\#\PP_1\les (\delta^{-s}\frac{\#\T}{\#\PP_1})^{2/3}\#\PP_1=\delta^{-2s/3}(\#\T)^{2/3}(\#\PP_1)^{1/3}.$$
	\end{proof}	
	\begin{corollary}
		\label{cjfhughurtghuirtgui}
		Assume $0\le s\le 1$. Let	$\sigma=\min(2s,2-2s)$.
		Consider a collection $\T$ of $\delta$-tubes and a collection $\PP$ of $\delta$-squares.
		
		Assume that  $\PP$ is a rectangular $(\delta,2s)$-KT set for each $p\in\PP$. Assume that $Y'(p)=\{T\in\T:\;p\cap T\not=\emptyset\}$ is a  $(\delta,\sigma)$-KT set.
		Then
		$$\sum_{T\in \T}(\#\{p\in\PP:\;p\cap T\not=\emptyset\})^2\les \delta^{-s}\#\PP\;(\max_{p\in\PP}\#Y'(p))^{1/2}.$$
	\end{corollary}
	\begin{proof}For $1\le 2^n\lesssim \delta^{-2s}$,  let $\T_n$ be the collection of $T\in\T$ for which $\#\{p\in\PP:\;p\cap T\not=\emptyset\}\sim 2^n$. We need to prove that 
		\begin{equation}
			\label{weejregu ug9ih9i590}
			2^{2n}\#\T_n\les \delta^{-s}\#\PP\;(\max_{p\in\PP}\#Y'(p))^{1/2}.
		\end{equation}
		We apply Theorem \ref
		{huufwf[gmiu9i96=y]} to the pair $(\T_n,\PP)$ and find 
		$$2^n\#\T_n\sim I(\T_n,\PP)\les \delta^{-2s/3}(\#\T_n)^{2/3}(\#\PP)^{1/3}, $$
		or
		$$2^{3n}\#\T_n\les \delta^{-2s}\#\PP.$$
		On the other hand,
		$$2^n\#\T_n\sim I(\T_n,\PP)\le \#\PP\;\max_{p\in\PP}\#Y'(p).$$
		We take the geometric average to get \eqref{weejregu ug9ih9i590}.
	\end{proof}
	We are now ready to prove Theorem \ref{jsdhfuurt gi  topgi yhi}. 
	\begin{proof}Since $s\le \frac23$ it follows that $\sigma=\min(2s,2-2s)\ge s$, so $Y'(p)$ is also  $(\delta,\sigma)$-KT.
		Corollary \ref{cjfhughurtghuirtgui} together with 
		\eqref{ juryegu rtugy t8ug8tug-9i90} and \eqref{ juryegu rtugy t8ug8tug-9i902} show that 
		$$\sum_{T\in\T}(\#\{p\in\PP:\;p\cap T\not=\emptyset\})^2\les \delta^{-s}\#S_1\#S_2(\# S_3)^{1/2}.$$
		Due to \eqref{fki9hik09hi-h0}, we thus have
		$$\E_{3,\delta}(S_1,S_2,S_3)\les \delta^{-s}\#S_1\#S_2(\# S_3)^{1/2}.$$
		All our arguments work for permutations of the three sets. Averaging the three estimates leads to the desired inequality
		$$\E_{3,\delta}(S_1,S_2,S_3)\les \delta^{-s}(\#S_1\#S_2\# S_3)^{5/6}.$$

	\end{proof}
	
	\begin{remark}At its heart, the proof of Theorem \ref{jsdhfuurt gi  topgi yhi} relied on a seemingly magical reduction of energy estimates to counting incidences between squares and tubes. There is yet another such reduction for the parabola, observed in \cite{O2} and \cite{SolSz}. We explain it below. 
		
		Let $S_1,S_2,S_3$ be as in Theorem \ref{jsdhfuurt gi  topgi yhi}. Let $\cP$ be a minimal collection of $\delta$-squares covering $S_1+S_2$.
		Note that
		\begin{equation}
			\label{djcuirheureugutgurt8uh98u98yu}
			\E_{3,\delta}(S_1,S_2,S_3)\sim \sum_{s\in \cP+S_3}(\#\{p\in\cP:\;p\cap (s-S_3)\not=\emptyset\})^2.
		\end{equation}
		The bi-Lipschitz function
		$$\Psi(x,y)=(x,x^2-y)$$
		maps translates of $\PP^1$ to lines. When applying this transformation, the collection of $\delta$-neighborhoods of translates of the parabola $s+\PP^1$, $s\in\cP+S_3=S_1+S_2+S_3+O(\delta)$, becomes a set $\T$ of $\delta$-tubes,  and $\cP$ becomes a set $\PP$ of $\delta$-squares (roughly speaking). We recast \eqref{djcuirheureugutgurt8uh98u98yu} as saying
		$$\E_{3,\delta}(S_1,S_2,S_3)\lesssim\sum_{T\in\T}\;(\#\{p\in\PP:\;p\cap T\not=\emptyset\})^2.$$
		The map $(x_1,x_2)\to (x_1+x_2,x_1^2+x_2^2)$ is bi-Lipschitz on $J_1\times J_2$, if $\dist(J_1,J_2)\sim 1$. Using this and a variant of Proposition \ref{k jg  tgu hu-uh49-t09359056 h9876} and Corollary \ref{idcj uhuireyygiytigy}, it can be shown  that $\PP$ is rectangular $(\delta,2s)$-KT. Moreover, $\#\PP\lesssim\#S_1\#S_2$.
		
		We can also see that $Y'(p)=\{T\in\T:\;p\cap T\not=\emptyset\}$ is $(\delta,s)$-KT and $\#Y'(p)\lesssim \#S_3$, for each $p\in\PP$. Indeed, assume $p=\Psi(s_1+s_2)+O(\delta)$, $T=\Psi(s_1'+s_2'+s_3'+\PP^1)+O(\delta)$. Then $p\cap T\not=\emptyset$ is equivalent to 
		$$s_1+s_2-(s_1'+s_2'+s_3')\in \cN_{O(\delta)}(\PP^1).$$
		Note that $s_3'$ determines $s_1',s_2'$.
		
		The estimate
		$$\E_{3,\delta}(S_1,S_2,S_3)\les \delta^{-s}(\#S_1\#S_2\#S_3)^{5/6}$$
		is now a consequence of Corollary \ref{cjfhughurtghuirtgui} and its symmetric versions. 
	\end{remark}

	We next show how trilinear energy estimates imply trilinear estimates for measures.  Given a measure $\mu$ supported on $\PP^1$, we will use the same notation $\mu$ for its pushforward (a measure on [-1,1]) via the projection map. We write
	$$\widehat{\mu}(x_1,x_2)=\int_{-1}^1e(\xi x_1+\xi^2 x_2)d\mu(\xi).$$
	\begin{proposition}
		\label{p o4i9i59ui=0g6-3}Let $s>0$.
		Let $J_1,J_2,J_3\subset [-1,1]$ be intervals separated by $\sim 1$. Fix $\delta>0$. Let $S_i$ be $(\delta,s)$-KT sets of points in $\{(\xi,\xi^2):\;\xi\in J_i\}$. 
		Assume the uniform estimate
		over all such collections    \begin{equation}
			\label{jfupuggu0956u0-04hui}
			\E_{3,\delta}(S_1,S_2,S_3)\les \delta^{-s} (\#S_1\#S_2\#S_3)^{5/6}.
		\end{equation}
		Let $R=\delta^{-1}$. Then the estimate
		$$\int_{B_R}|\widehat{\mu_1}|^2|\widehat{\mu_2}|^2|\widehat{\mu_3}|^2\les R^{2-\frac{5s}2}(\|\mu_1\|\|\mu_2\|\|\mu_3\|)^{5/6}$$
		holds for each $\mu_i$ supported on $\{(\xi,\xi^2):\;\xi\in J_i\}$, that satisfy \eqref{e4}.
	\end{proposition}
	
	\begin{proof}
		Partition $[-1,1]$ into intervals $I$ of length $1/R$. For a non-negative integer $l$, let $\cI_i(l)$ be those intervals $I_i$ inside $J_i$ such that $\mu_i(I_i)\sim 2^{-l}(1/R)^s$. Letting $S_i(l)$ be the points $(c_i,c_i^2)$, with $c_i$ the center of $I_i\in\cI_i(l)$, we find that
		$S_i(l)$ is a $(\delta,s,2^{l})$-KT set with $\# S_i(l)\lesssim \|\mu_i\|2^lR^s$.
		
		We use the fact that each $(\delta,s,C)$-KT set can be partitioned into $\les C$ many $(\delta,s)$-KT sets. See Lemma 2.8 in \cite{demwangszem}. This allows us  to partition each $\cI_i(l)$ into collections $\cI_{i,k}(l)$, $k\les 2^l$, whose centers form  $(\delta,s)$-KT sets $S_{i,k}(l)$. We write $\mu_{I_i}$ for the restriction of $\mu$ to $I$.
		
		Write
		$$\mu_{i}=\sum_{l\ge 0}\sum_{k\les 2^l}\sum_{I_i\in\cI_{i,k}(l)}\mu_{I_i}.$$
		There are $O(\log R)$ many values of $l$ contributing to the summation.
		Let $\eta_R$ be smooth, such that $\widehat{\eta_R}\ge 1_R$ and ${\eta_R}$ is supported inside $B(0,\frac1{R})$. Then using Cauchy--Schwarz we find 
		\begin{align*}
			\int_{B_R}|\widehat{\mu_1}|^2|\widehat{\mu_2}|^2|\widehat{\mu_3}|^2&\les\sum_{l_1,l_2,l_3}2^{l_1+l_2+l_3}\sum_{k_i\les 2^{l_i}}\int_{\R^2}\prod_{i=1}^3|\sum_{I_i\in\cI_{i,k_i}(l_i)}\widehat{\mu_{I_i}*\eta_R}|^2\\&=\sum_{l_1,l_2,l_3}2^{l_1+l_2+l_3}\sum_{k_i\les 2^{l_i}}\int_{\R^2}\prod_{i=1}^3\sum_{I_i,I_i'\in\cI_{i,k_i}(l_i)}\widehat{\mu_{I_i}*\eta_R}\;\overline{\widehat{\mu_{I_i'}*\eta_R}}.  
		\end{align*}
		We estimate one term in the summation using Plancherel's formula,
		$$\int_{\R^2}\widehat{\mu_{I_1}*\eta_R}\;\widehat{\mu_{I_2}*\eta_R}\;\widehat{\mu_{I_3}*\eta_R}\;\overline{\widehat{\mu_{I_1'}*\eta_R}}\;\overline{\widehat{\mu_{I_2'}*\eta_R}}\;\overline{\widehat{\mu_{I_3'}*\eta_R}}=$$
		$$\int_{\R^2}(\mu_{I_1}*\eta_R)*(\mu_{I_2}*\eta_R)*(\mu_{I_3}*\eta_R)\times (\mu_{I_1'}*\eta_R)*(\mu_{I_2'}*\eta_R)*(\mu_{I_3'}*\eta_R).$$
		If $I_i$ is centered at $c_i$, then $(\mu_{I_1}*\eta_R)*(\mu_{I_2}*\eta_R)*(\mu_{I_3}*\eta_R)$ is supported on $(c_1,c_1^2)+(c_2,c_2^2)+(c_3,c_3^2)+O(1/R)$. Also, $\|\mu_{I_i}*\eta_R\|_\infty\lesssim R^{2-s}2^{-l_i}$. 
		
		Letting $S_{i,k_i}(l_i)$ be the points $(c_i,c_i^2)$ with $c_i$ the centers of $I_i\in\cI_{i,k_i}(l_i)$ we find that
		$$\int_{B_R}|\widehat{\mu_1}|^2|\widehat{\mu_2}|^2|\widehat{\mu_3}|^2\les  R^{-2}R^{4-6s}\sum_{l_1,l_2,l_3}2^{-l_1-l_2-l_3}\sum_{k_1,k_2,k_3}\E_{3,1/R}(S_{1,k_1}(l_1),S_{2,k_2}(l_2),S_{3,k_3}(l_3)).
		$$
		Combining  our hypothesis \eqref{jfupuggu0956u0-04hui} with Cauchy--Schwarz and the fact that
		$$\sum_{k_i}\#S_{i,k_i}(l_i)=\#S_{i}(l_i)\lesssim 2^{l_i}R^s\|\mu_i\|,$$we find 
		\begin{align*}
			\int_{B_R}|\widehat{\mu_1}|^2|\widehat{\mu_2}|^2|\widehat{\mu_3}|^2&\les \delta^{-s}R^{-2}R^{4-6s}\sum_{l_1,l_2,l_3}2^{-l_1-l_2-l_3}\sum_{k_1,k_2,k_3}(\#S_{1,k_1}(l_1)\#S_{2,k_2}(l_2)\#S_{3,k_3}(l_3))^{5/6}\\&\les\delta^{-s}R^{-2}R^{4-6s}\sum_{l_1,l_2,l_3}2^{-\frac{5(l_1+l_2+l_3)}6}(\sum_{k_1,k_2,k_3}\#S_{1,k_1}(l_1)\#S_{2,k_2}(l_2)\#S_{3,k_3}(l_3))^{5/6}\\&\lesssim \delta^{-s}R^{-2}R^{4-6s}\sum_{l_1,l_2,l_3}R^{5s/2}(\|\mu_1\|\|\mu_2\|\|\mu_3\|)^{5/6}\\& \les R^{2-\frac{5s}2}(\|\mu_1\|\|\mu_2\|\|\mu_3\|)^{5/6}.
		\end{align*}

	\end{proof}

	We next present a trilinear to linear reduction, also known as a broad-narrow argument. Such arguments go back to \cite{BoGu}.
	\begin{proposition}
		\label{io ejfrefiugu-9 5ih905uih9}
		Let $s>0$. 
		Assume the uniform estimate
		$$\int_{B_R}|\widehat{\mu_1}|^2|\widehat{\mu_2}|^2|\widehat{\mu_3}|^2\le T_{K,\epsilon} R^{2-\frac{5s}2+\epsilon}(\|\mu_1\|\|\mu_2\|\|\mu_3\|)^{5/6}$$
		holds for each $ R\ge K^2\ge 1$, each $\epsilon>0$, each $\mu_i$ supported on $\{(\xi,\xi^2):\;\xi\in J_i\}$ and satisfying \eqref{e4}, for each intervals $J_i\subset [-1,1]$ separated by $\ge 1/K$.
		
		Then the estimate
		$$\int_{B_R}|\widehat{\mu}|^6\les R^{2-\frac{5s}2}\|\mu\|$$
		holds for  $R\ge 1$ and each  $\mu$ supported on $\PP^1$ that satisfies \eqref{e4}.
	\end{proposition}
	\begin{proof}
		Call $C_R$ the smallest constant such that
		$$\int_{B_R}|\widehat{\mu}|^6\le C_R\|\mu\|.$$
		Fix $K\ge 1$, to be chosen later. Partition $[-1,1]$ into intervals $I\in\cI_k$ of length $1/K$. For $I,I'\in\cI_K$ we write $I\not\sim I'$ if $I$ is not adjacent to $I'$.
		Then for each $x$
		$$|\widehat{\mu}(x)|^6\le 100\max_{I\in\cI_K}|\widehat{\mu_I}(x)|^6+K^{100}\max_{I_1\not\sim I_2\not\sim I_3}|\widehat{\mu_{I_1}}(x)|^2|\widehat{\mu_{I_2}}(x)|^2|\widehat{\mu_{I_3}}(x)|^2.$$
		Let $R\ge K^2$. Integration shows that
		$$\int_{B_R}|\widehat{\mu}|^6\le 100\int_{B_R}\sum_{I\in\cI_K}|\widehat{\mu_I}|^6+K^{200}\max_{I_1\not\sim I_2\not\sim I_3}\int_{B_R}|\widehat{\mu_{I_1}}|^2|\widehat{\mu_{I_2}}|^2|\widehat{\mu_{I_3}}|^2.$$
		We first observe that since $\|\mu\|\lesssim 1$
		$$\max_{I_1\not\sim I_2\not\sim I_3}\int_{B_R}|\widehat{\mu_{I_1}}|^2|\widehat{\mu_{I_2}}|^2|\widehat{\mu_{I_3}}|^2\le T_{K,\epsilon}R^{2-\frac{5s}2+\epsilon}(\|\mu_{I_1}\|\|\mu_{I_2}\|\|\mu_{I_3}\|)^{5/6}\lesssim T_{K,\epsilon}R^{2-\frac{5s}2+\epsilon}\|\mu\|.$$
		We next analyse the first term. Given $I=[a-\frac1{2K},a+\frac1{2K}]$, let  $\nu^I$ be the measure on $[-1,1]$ given by
		$$\nu^I(A)=K^s\mu_I(\frac{A}{2K}+a).$$
		Note that $\nu^I$ satisfies \eqref{e4}, and $\|\nu^I\|=K^s\mu(I)$.
		Moreover, via the change of variables $\eta=2K(\xi-a)$ we find 
		\begin{align*}  
			|\widehat{\mu_I}(x_1,x_2)|=&|\int_{[a-\frac1{2K},a+\frac1{2K}]}e(\xi x_1+\xi^2 x_2)d\mu(\xi)|\\=&\frac1{K^s}|\int_{[-1,1]}e((a+\frac{\eta}{2K})x_1+(a+\frac{\eta}{2K})^2x_2)d\nu^I(\eta)|\\=&\frac1{K^s}|\widehat{\nu^I}(\frac{x_1+2a x_2}{2K},\frac{x_2}{4K^2})|.
		\end{align*}
		The image of $B_R$ under the map $(x_1,x_2)\mapsto(\frac{x_1+2a x_2}{2K},\frac{x_2}{4K^2})$ may be covered with $\sim K$ many balls $B_{R/K^2}$ with radius $R/K^2$.
		It follows that 
		$$\int_{B_R}|\widehat{\mu_I}|^6\lesssim K^{4-6s}\max_{B_{R/K^2}}\int_{B_{R/K^2}}|\widehat{\nu^I}|^6,$$
		Our hypothesis implies that
		$$\int_{B_{R/K^2}}|\widehat{\nu^I}|^6\le C(R/K^2)K^{s}\mu(I).$$
		It follows that 
		$$\int_{B_R}\sum_{I\in\cI_K}|\widehat{\mu_I}|^6\lesssim C(R/K)K^{4-5s}\|\mu\|.$$
		Putting things together we get for each $\epsilon>0$
		$$C_R\le AC(R/K^2)K^{4-5s}+B_{K,\epsilon}R^{2-\frac{5s}2+\epsilon}.$$
		Here $A\ge 1$ is an absolute constant, independent of $K,R,\epsilon$.
		Writing $D_R=C_R/R^{2-\frac{5s}2+\epsilon}$ we find that 
		$$D_R\le AK^{-2\epsilon}D_{R/K^2}+B_{K,\epsilon}.$$
		We fix $\epsilon$ as small as small as we wish and let $K=A^{1/2\epsilon}$. We then iterate this inequality $n=[\log R/2\log K]$ many times. Since $R/K^{2n}\sim 1$, we find
		$D_R\lesssim_\epsilon \log R.$
		Since $\epsilon$ is arbitrary, we conclude that $C_R\les R^{2-\frac{5s}{2}}$.
		
	\end{proof}

	We close this subsection by observing  that Theorem \ref{io jfpurei90 h9=h} is an immediate consequence of Theorem \ref{jsdhfuurt gi  topgi yhi} (see also Remark \ref{chcurt ug8u-g06u9=67u0978}), Proposition \ref{p o4i9i59ui=0g6-3} and Proposition \ref{io ejfrefiugu-9 5ih905uih9}.

	\subsection{The decoupling method}
	We prove Theorem \ref{jki regu rthu8 yuh 60-u4j}, that we recall below.
	\begin{theorem}
		
		Let $\frac12\le s\le 1$. Assume $\mu$ is supported on a curve $\Gamma$ with nonzero curvature, and satisfies \eqref{e4 giobjytibuiuiu}.
		Then
		\begin{equation}
			\label{4ui37854yijvout9gi5 hi60-ji}
			\|\widehat{\mu}\|^6_{L^6(B_R)}\les R^{1-s+\beta},\;\;\;\beta=\frac{1-s^2}{2s+1}.
		\end{equation}
		If $\mu$ is AD regular then
		\begin{equation}
			\label{ejfiorugtugur8huhfryygyt}
			\|\widehat{\mu}\|^6_{L^6(B_R)}\les R^{1-s+\beta-c_\mu},\;\;\;\beta=\frac{1-s}{2},
		\end{equation}
		when $c_\mu>0$ depends on $\mu$.
	\end{theorem}   
	
	\begin{remark}
		Here is a comparison with the previous two methods. Our new exponent  in \eqref{4ui37854yijvout9gi5 hi60-ji} is smaller than the one in  \eqref{kjrefut0=or0i569y0-6iu09} precisely when $s<\frac{1+\sqrt{21}}{10}=0.558...$.
		
		In the AD-regular case, the exponent $1-s+\beta$ in  \eqref{ejfiorugtugur8huhfryygyt}  is smaller than the one in  \eqref{kjrefut0=or0i569y0-6iu09} precisely when $s<\frac{3}5=0.6$.
		
		The exponent in \eqref{4ui37854yijvout9gi5 hi60-ji} is always bigger than the one in \eqref{cjfgutgutugtuhu-h9u} for the parabola.
		
		In the AD-regular case, the exponent $1-s+\beta$ in  \eqref{ejfiorugtugur8huhfryygyt}  is bigger than the one in \eqref{cjfgutgutugtuhu-h9u} for the parabola, with equality when $s=1/2$.
	\end{remark}
	\smallskip
	
	Let us now see the proof of the theorem. We recall the framework in  \cite{demwangszem}.
	
	Fix a smooth $\psi:\R^2\to \R$ satisfying $0\le \psi\le 1_{B(0,1)}$. Write $\psi_\delta(\xi)=\delta^{-2}\psi(\frac{\xi}{\delta})$ and $\mu_\delta=\mu*\psi_\delta$. 
	Here is  some equivalent numerology. Let $\beta>0$. The estimate
	$$\|\widehat{\mu}\|^6_{L^6(B_R)}\les R^{1-s+\beta}$$
	will be recasted as either
	$$\|\widehat{\mu_{1/R}}\|_{L^6(\R^2)}^6\les R^{1-s+\beta}$$
	or
	\begin{equation}
		\label{jcfvtg8gu8u096849}
		\|F\|_{L^6}^6\les R^{5s-11+\beta},
	\end{equation}
	using a rescaled function defined below. 
	\medskip

	Since $\mu_{1/R}$ is supported on the $1/R$-neighborhood of $\Gamma$,
	we partition this neighborhood into essentially rectangular regions $\theta$ with dimensions roughly $R^{-1/2}$ and $1/R$. We also partition the $1/R^{1/2}$-neighborhood into  essentially rectangular regions $\tau$ with dimensions roughly $R^{-1/4}$ and $R^{-1/2}$. Each $\theta$ sits inside some $\tau$.
	\smallskip
	
	Let $D,M,P$ be dyadic parameters, with $D$ and $P$  integers. We decompose $\mu$
	$$\mu=\sum_{D,M,P}\mu_{D,M,P},$$
	with each $\mu_{D,M,P}$ satisfying the following properties:
	\\
	\\
	(P1) for each $\theta$, either $\mu_{D,M,P}(\theta)=\mu(\theta)\sim MR^{-s}$, or  $\mu_{D,M,P}(\theta)=0$. We call $\theta$ {\em active} if it falls into the first category.
	\\
	\\
	(P2) each $\tau$ contains either $\sim P$ or no active $\theta$. We call $\tau$ {\em active} if it fits into the first category.
	\\
	\\
	(P3) there are $\sim D$ active $\theta $, or equivalently, there are $\sim D/P$ active $\tau $.
	\medskip
	
	We note that these properties together with \eqref{e4} force the following inequalities
	\begin{equation}
		\label{e5}
		M\lesssim R^{s/2}
	\end{equation}
	\begin{equation}
		\label{e6}
		DM\lesssim R^s
	\end{equation}
	\begin{equation}
		\label{e7}
		MP\lesssim R^{3s/4}.
	\end{equation}
	\smallskip

	For the rest of the argument we fix $D,M,P$ and let $F=R^{s-2}\widehat{\mu_{D,M,P}*\psi_{1/R}}$ be the $L^\infty$ upper normalized version, $\|F\|_\infty\lesssim 1$. We now recall the relevant estimates from \cite{demwangszem}. Invoking the triangle inequality, it suffices to prove that
	\begin{equation}
		\label{e3}
		\|F\|_{L^6}^6\les R^{\frac{9s^2+s-1}{2s+1}-9}=R^{5s-11+\beta},\;\;\;\beta=\frac{1-s^2}{2s+1}.
	\end{equation}
	Write $\mu_\theta$ for the restriction of $\mu$ to an active  $\theta$.
	Let $F_\theta=R^{s-2}\widehat{\mu_\theta*\psi_{1/R}}$ be the Fourier restriction of $F$ to $\theta$, so that
	$F=\sum_\theta F_\theta$.
	Then
	$$F_\theta=\sum_{T\in\cT_\theta}\langle F_\theta,W_T\rangle W_T+O(R^{-100}).$$
	Each $T$ is a rectangle (referred to as tube) with dimensions $R^{1/2}$ and $R$, with the long side pointing in the direction normal to $\theta$. The term $O(R^{-100})$ is negligible, and can be dismissed. Its role is to ensure that all $T$ sit inside, say,  $[-R,R]^2$. The wave packet $W_T$ is $L^2$ normalized, has spectrum inside (a slight enlargement of) $\theta$, and is essentially concentrated spatially in (a slight enlargement of) $T$. Thus, with $\chi_T$ being a smooth approximation of $1_T$, we have
	$$|W_T|\lesssim R^{-3/4}\chi_T.$$

	Given the dyadic parameter $\lambda$ we write
	$$\cT_{\lambda,\theta}=\{T\in\cT_\theta:\;|\langle F_\theta,W_T\rangle|\sim \lambda\}.$$
	We have
	\begin{equation}
		\label{e15}
		\lambda\le \lambda_{max}\sim MR^{-\frac54}.
	\end{equation}
	\eqref{e3} boils down to proving
	\begin{equation}
		\label{e3gtprgptphyphyt==}
		\|\sum_{T\in \cup_\theta\cT_{\lambda,\theta}}\langle F_\theta,W_T\rangle W_T\|_{L^6}^6\les R^{\frac{9s^2+s-1}{2s+1}-9}.
	\end{equation}
	For each rectangle $B$ with dimensions $\Delta$ and $R$, and with orientation identical to that of the tubes in $\cT_\theta$, we have  the estimate
	\begin{equation}
		\label{e12}
		\#\{T\in\cT_{\lambda,\theta}:\; T\subset B\}\lesssim (\frac{\Delta}{\sqrt{R}})^{1-s}\frac{MR^{\frac{s-5}{2}}}{\lambda^2}.
	\end{equation}
	In particular,
	\begin{equation}
		\label{e11}
		\#\cup_{\theta}\cT_{\lambda,\theta}\lesssim \frac{MD}{\lambda^2R^2}.
	\end{equation}
	Each arc of length $r\lesssim 1$ on $\Gamma$ intersects at most
	\begin{equation}
		\label{rifururgurtgrthi90hi90}
		\lesssim (r\sqrt{R})^s\frac{R^{s/2}}{M}\end{equation}
	many $\theta$.
	
	We derive two estimates.
	\\
	\\
	\textbf{First estimate:} We first use decoupling (\cite{BD}) into caps $\theta$ (combined with property (P3)), then \eqref{e11} to get
	\begin{align*}
		\|\sum_{\theta\text{ active}}\sum_{T\in\cT_{\lambda,\theta}}\langle F_\theta,W_T\rangle W_T\|_6^6&\les D^2\sum_{\theta\text{ active}}\|\sum_{T\in\cT_{\lambda,\theta}}\langle F_\theta,W_T\rangle W_T\|_6^6\\&\les D^2R^{-9/2}\lambda^6\sum_{\theta\text{ active}}\|\sum_{T\in\cT_{\lambda,\theta}}\chi_T\|_6^6\\&\les D^2R^{-9/2}\lambda^6R^{3/2}\#(\cup_{\theta}\cT_{\lambda,\theta})\\&\les \frac{MD^3\lambda^4}{R^{5}}.
	\end{align*}Using  \eqref{e15} we conclude with
	\begin{equation}
		\label{oirpofiigtigihoyt0ho0}
		\|\sum_{\theta\text{ active}}\sum_{T\in\cT_{\lambda,\theta}}\langle F_\theta,W_T\rangle W_T\|_6^6\les\frac{M^5D^3}{R^{10}}.\end{equation}
	\textbf{Second estimate:}
	\begin{defn}Consider a collection consisting of $\sim r'/P'$ many caps $\tau$, each of which contains $\sim P'$ many $\theta$. Call $\cD$ the collection of all these $\theta$, and note that $\#\cD\sim r'$.  	
		Let $\cB\subset \cup_{\theta\in\cD}\cT_{\lambda,\theta}$ be a collection of tubes $T$  intersecting a fixed  $\sqrt{R}$-square $q$. Note that there are $O(1)$ many possible $T$ for each $\theta\in\cD$.
		
		We call $\cB$ a bush with data $(P',r',q)$.
	\end{defn}	
	We recall the following bush estimate, that uses decoupling into the larger caps $\tau$.

	\begin{prop}
		\label{p5}
		Let $\cB$ be a bush with data $(P',r',q)$.  Then
		\begin{equation}
			\label{dokoregotrgo[p]oh}
			\|\sum_{T\in \cB}\langle F_\theta,W_T\rangle W_T\|_{L^6(q)}^6\les R^{-7/2}(r')^3(P')^2\lambda^6.\end{equation}
	\end{prop}
	Now, for each dyadic $r\ge 1$, let
	$$\cQ_r=\{q:\;q\text { intersects }\sim r\text{ tubes in } \cup_\theta\cT_{\lambda,\theta}\}.$$
	To estimate $|\cQ_r|$, we note that due to \eqref{e12} and \eqref{rifururgurtgrthi90hi90}, the rescaled copies of the tubes $T\in \cup_\theta\cT_{\lambda,\theta}$ satisfy the hypotheses of Corollary \ref{pout98u8huy8ht-8h} with $\delta=R^{-1/2}$, $K_1\sim \frac{R^{s/2}}{M}$ and $K_2\sim \frac{MR^{\frac{s-5}{2}}}{\lambda^2}.$ Combining this  with \eqref{e11} shows that
	\begin{equation}
		\label{dokoregotrgo[p]ohhpoihopyih}
		\#\cQ_r\les \frac{R^{\frac12}}{r^{\frac{s+1}{s}}}(\frac{R^{s-\frac52}}{\lambda^2})^{\frac1s}\frac{MD}{\lambda^2R^2}.
	\end{equation}
	For each $q\in\cQ_r$, the collection $\cB(q)=\{T\in \cup_\theta\cT_{\lambda,\theta}:\;q\cap T\not=\emptyset\}$ can be partitioned into $\lessapprox 1$ many bushes with data $(P',r',q)$ satisfying $P'\le P$, $r'\le r$. Thus, combining \eqref{dokoregotrgo[p]oh} with \eqref{dokoregotrgo[p]ohhpoihopyih} we find
	$$\|\sum_{T\in \cup_\theta\cT_{\lambda,\theta}}\langle F_\theta,W_T\rangle W_T\|_{L^6(\cup_{q\in\cQ_r q})}^6\les R^{-7/2}P^2\lambda^6R^{\frac12}(\frac{R^{s-\frac52}}{\lambda^2})^{\frac1s}\frac{MD}{\lambda^2R^2}r^{2-\frac1s}$$$$=R^{-4-\frac5{2s}}MDP^2\lambda^{4-\frac2s}r^{2-\frac1s} .$$
	Using that $\lambda\lesssim M R^{-\frac54}$ and $r\lesssim D$ and since there are $\lessapprox 1$ dyadic values of $r$, we get our second estimate
	\begin{equation}
		\label{lkporkforigoitpgoitopgi[tpo]}
		\|\sum_{T\in \cup_\theta\cT_{\lambda,\theta}}\langle F_\theta,W_T\rangle W_T\|_{L^6}^6\les R^{-9}M^{5-\frac2s}D^{3-\frac1s}P^2.
	\end{equation}
	
	If $\mu$ is AD-regular then $M,\;D\sim R^{s/2}$, $P\sim R^{s/4}$. In this case we get
	$$
	\|\sum_{T\in \cup_\theta\cT_{\lambda,\theta}}\langle F_\theta,W_T\rangle W_T\|_{L^6}^6\les R^{-10-\frac12+\frac{9s}{2}}=R^{5s-11+\beta-c_\mu},\;\;\;\beta=\frac{1-s}{2}
	$$
	and thus we get \eqref{jcfvtg8gu8u096849} with $\beta=\frac{1-s}{2}-c_\mu$. See Remark 9.3 in \cite{demwangszem} for the role of $c_\mu$.

	But when $M\ll R^{s/2}$, \eqref{lkporkforigoitpgoitopgi[tpo]} by itself gives no useful upper bound. Instead, we take the weighted geometric average  $\eqref{oirpofiigtigihoyt0ho0}^{1-\alpha} \eqref{lkporkforigoitpgoitopgi[tpo]}^{\alpha}$. When choosing $\alpha$, the goal is to end up with an exponent for $M$ that is larger then the sum of the exponents of $D$ and $P$, since \eqref{e6} and \eqref{e7} only give upper bounds for $MD$ and $MP$, rather than for $D$ and $P$. This leads to the restriction $\alpha\le \frac{2s}{2s+1}$. Computations show that $\alpha= \frac{2s}{2s+1}$ leads to the best estimate, namely, 
	$$\|\sum_{T\in \cup_\theta\cT_{\lambda,\theta}}\langle F_\theta,W_T\rangle W_T\|_{L^6}^6\les (DM)^{\frac{1+6s}{2s+1}}(PM)^{\frac{4s}{2s+1}}R^{-9-\frac1{2s+1}}\les R^{\frac{9s^2+s-1}{2s+1}-9}.$$
	The same estimate for $\|F\|_{L^6}^6$ follows by summing over all relevant parameters.

	\section{Sum-product problems}
	\subsection{Projection results}
	For the proof of Theorem \ref{thm.dis}, we require a discretised projection theorem. Let $\pi_x$ be the map $(a,b) \mapsto a + xb.$ In this subsection $\sigma = \min\{s,2-s\}.$
	\begin{theorem}\label{thm.projdis}
		Let $\e >0,$ $0 < s < 2$ and let $s_2\ge s_1> 0$ be such that $s_1 + s_2 = s.$ Let $K_1,K_2,K_3 \geq 0.$ 
		
		\indent Let $\cP \subset [0,1]^2$ be a $(\de,s_1,s_2,K_1,K_2)$-quasi-product set. Let $E \subset [0,2]$ be a $\delta$-separated $(\delta,\sigma,K_3)$-KT set. There is $F\subset E$ with $\#F\ge  (1-\delta^{O(\e)})\#E$ such that for each  $x \in F$ we have 
		\begin{equation}\label{eq.projresult}
			\cn{\pi_x(\cQ)} \gtrsim  K_1^{-1}K_2^{-1}K_3^{-1/2}\de^{O(\e)}\delta^{s/2}\sqrt{\#E}\#\cQ \text{ for all } \cQ \subset \cP \text{ with } \#\cQ > \de^\e \#\cP.
		\end{equation}
		\begin{proof}
			It will suffice to prove this result for a single $x \in E.$ Suppose we are indeed in possession of such a result. Set $E_1 := E,$ and let $x_1 \in E_1$ be such that \eqref{eq.projresult} holds. Set $E_2 := E_1 \setminus \{x_1\},$ noting that $E_2$ (indeed any subset) will retain its $(\de,\sigma,K_3)$-structure. At stage $n,$ after repeatedly applying our result, we will still be able to to obtain \eqref{eq.projresult}, provided that $\#E_n \gtrsim \dep \# E.$ This will cease to be true some iterations after stage $N$: when 
			\begin{equation}
				\# E_N = \# E - N \sim \dep \#E. 
			\end{equation}
			In particular, 
			\begin{equation}
				N = (1-\dep) \#E .
			\end{equation}
			So select $F := \{x_1,\ldots,x_N\}.$
			
			We proceed to now prove the result for a solitary $x \in E.$
			Let $M$ be such that for all $x \in E$ there is   $Q_x \subset \cP$ with $\#Q_x > \de^\e\#\cP$ and
			\begin{equation}\label{eq.contrad}
				\cn{\pi_x(Q_x)} \le M\sqrt{\#E}\#\cQ_x\;.
			\end{equation}
			We need to prove that 
			\begin{equation}
				\label{eq.cont2hhhh}
				M\gtrsim K_1^{-1}K_2^{-1}K_3^{-1/2}\de^{O(\e)}\delta^{s/2}.
			\end{equation}
			Let $\cT_x$ be the tubes $\pi_x^{-1}(I) \cap [0,1]^2$ for all $I \in \cD_\de(\pi_x(\cQ_x)), x \in E.$ Set $\cT := \cup_{x \in E}\cT_x.$ In particular, from \eqref{eq.contrad}, we know that
			\begin{align}
				\#\cT  &\le M\sqrt{\#E}\sum_{x \in E}\#\cQ_x \\
				&\le M\#E^{3/2}\#\cP. \label{eq.cont2}
			\end{align}
			Since
			\begin{equation}
				\sum_{p \in \cP} \#\{x \in E: p \in \cQ_x\} = \sum_{x \in E} \#\cQ_x > \de^\e\#\cP\#E, 
			\end{equation}
			we may find $\cQ \subset \cP$ with $\#\cQ > \de^\e\#\cP/2$ so that for all $q \in \cQ,$ we have
			\begin{equation}
				\# \{x \in E : q \in \cQ_x\}  \gtrsim \dep \# E.
			\end{equation}
			
			For each $q \in \cQ$ consider the shading $Y(q)$ consisting of the $\gtrsim \dep \# E$ tubes in $\cT$ which intersect $q.$ Since $E$ is a $(\delta,\sigma,K_3)$-set, so is each $Y(q), q \in \cQ.$ The set $\cQ$ remains a $(\de,s_1,s_2,K_1,K_2)$-quasi-product set, so we apply Proposition \ref{thm.product} in its dual form (with the family $\cQ$ playing the role of tubes) to obtain
			\begin{align}
				\# \cT  &\gtrsim \dep K_1^{-1}K_2^{-1}K_3^{-1/2}\#E^{1/2}\delta^{s/2}\sum_{q \in Q} \# Y(q)\\
				&\gtrsim \dep K_1^{-1}K_2^{-1}K_3^{-1/2}\#E^{3/2}\delta^{s/2}\#\cP.
			\end{align}
			This together with \eqref{eq.cont2} proves \eqref{eq.cont2hhhh}.
		\end{proof}
	\end{theorem}
	
	\begin{theorem}\label{thm.proj}
		Let $\e > 0.$ Let $0 < s < 2$ and let $s_2\ge s_1 \geq 0$ be such that $s_1 + s_2 = s.$ Let $0 < t \leq \sigma.$ 
		
		\indent Let $\cP \subset [0,1]^2$ be a $(\de,s_1,s_2,\de^{-\e},\de^{-\e})$-quasi-product set with $\#\cP > \delta^{-s+\e}.$  Let $E \subset [0,2]$ be a $(\delta,t,\delta^{-\e})$-set. There exists $F \subset E$ with $\#F > \delta^{O(\e)}\#E$ so that for all $x\in F$ we have
		\begin{equation}
			\label{jututgutgtugutpu}
			\cont{(s+t)/2}(\pi_x(\cP)) \gtrsim  \delta^{O(\e)}.
		\end{equation}
		\begin{proof}
			By Lemma \ref{lem.unifsubset}, we may replace $E$ and $\cP$ by  $\gtrsim \de^\e$-fractions which are $\e$-uniform. Assuming that $E$ is $\e$-uniform, we will find $F\subset E$ with $\#F\gtrsim \#E$ such that \eqref{jututgutgtugutpu} holds for each $x\in F$.
			
			By Proposition \ref{prop.properties}, for each $\delta \leq \rho \leq 1$, $\cP(\rho)$ is a $(\rho,s,\de^{-O(\e)})$-set, and $E(\rho)$ is a $(\rho, t, \de^{-O(\e)})$-set. In particular, $E(\rho)$ is a $(\rho,t,K_3)$-KT set with $K_3=\#E(\rho)\rho^t\delta^{-O(\e)}$. 
			
			For the scales $$\delta, 2\delta, \ldots, 2^{k}\de$$ where $k$ is selected so that $2^{k}\de = \delta^{O(\e)}$, for each $0\le j\le k$ we select $F_j \subset E$ with $\#F_j > (1-(2^j\de)^{\e})\#E$ by Theorem \ref{thm.projdis} so that \eqref{eq.projresult} holds at scale $\rho_j=2^{j}\de$, see \eqref{jhyufyyrefyg98gu98}. Now set 
			\begin{equation}
				F := \bigcap_{j}F_j.
			\end{equation}
			We have
			\begin{equation}
				\label{jvjfvujhfudvhugui}
				\#F > (1-\sum_{j}(2^{j}\de)^\e)\#E \gtrsim\#E.
			\end{equation}
			\indent 
			Also, for each $x\in F$, \eqref{eq.projresult} reads (since $\#\cP(\rho_j)\gtrsim \rho_j^{-s+O(\e)}$ and $K_3=\#E(\rho)\rho^t\delta^{-O(\e)}$)
			\begin{equation}
				N_{\rho_j}(\pi_x(\cQ^*))\gtrsim \delta^{O(\e)}\rho_j^{s/2}\rho_j^{-t/2}\rho_j^{-s}=\delta^{O(\e)}\rho_j^{-\frac{s+t}{2}},\;\forall\cQ^*\subset \cP(\rho) \text{ with } \#\cQ^*>\rho_j^\e\#\cP(\rho).
			\end{equation}
			Due to the uniformity of $\cP$, this is essentially the same as
			\begin{equation}
				\label{jhyufyyrefyg98gu98}
				N_{\rho_j}(\pi_x(\cQ))\gtrsim \delta^{O(\e)}\rho_j^{s/2}\rho_j^{-t/2}\rho_j^{-s}=\delta^{O(\e)}\rho_j^{-\frac{s+t}{2}},\;\forall\cQ\subset \cP \text{ with } \#\cQ\gtrsim \rho_j^\e\#\cP.
			\end{equation}
			Fix an arbitrary $x \in F$. In order to verify \eqref{jututgutgtugutpu}, we let $\cU$ be a cover of $\pi_x(\cP)$ by disjoint $\de$-adic intervals of length at least $\de$ such that
			$$\cont{(s+t)/2}(\pi_x(\cP))\sim \sum_{U\in\cU}\diam(U)^{\frac{s+t}{2}}.$$
			By pigeonholing, there is some $\rho = 2^j\de$, $j\ge 0$, such that
			\begin{equation}
				S :=  \bigcup_{U \in \cU: \diam(U) = \rho}\pi_x(\cP) \cap U,
			\end{equation} 
			satisfies \begin{equation}\label{eq.densepre}
				\#\cQ \gtrsim (\log\frac1\delta)^{-1}\#\cP,\end{equation}
			where
			$\cQ=\pi_x^{-1}(S)\cap \cP$.

			If $\rho > \de^{O(\e)},$ then $\cU$ contains at least one interval of length $> \delta^{O(\e)}$ and so
			\begin{equation}
				\cont{(s+t)/2}(\pi_x(\cP)) > \delta^{O(\e)}.
			\end{equation}
			If not, then $\rho=\rho_j$ for some $0\le j\le k$. Using \eqref{jhyufyyrefyg98gu98} and  \eqref{eq.densepre}
			\begin{equation}
				\#\{U\in\cU:\;\diam(U)=\rho\}= \cns{\rho}{S} \gtrsim \delta^{O(\e)}\rho^{-(s+t)/2},
			\end{equation}
			and so $\cont{(s+t)/2}(\pi_x(\cP)) \gtrsim\delta^{O(\e)}.$
		\end{proof}
	\end{theorem}

	\subsection{Proof of Theorem \ref{thm.dis}}
	\subsubsection{Pre-processing} Let $\delta, \epsilon >0$ be small parameters, which we may take smaller without being explicit. Let $A \subset [1/2,1]$ be a $(\delta,s,\delta^{-\e})$-set. This hypothesis will be implicit in all our results in this section. Some of the results will need $A$ to also be uniform, but that will be made explicit.  Let $\pi:A\times A \rightarrow \R$ be the map $(x,y) \mapsto y/x.$ Write $\cD_0 := \cD_\delta(A/A).$ Since
	\begin{equation}
		\sum_{r \text{ dyadic}}\sum_{I \in \cD_0\atop{ \#\pi^{-1}(I) \sim r}} \#\pi^{-1}(I) = \#A^2,
	\end{equation}
	we may find some dyadic $r_0$ such that for 
	\begin{equation}
		\label{juhuyuyyg89urtg8rt8g8rt8gh-}
		\cD:= \{I \in \cD: \#\pi^{-1}(I) \sim r_0\}
	\end{equation}
	we have
	\begin{equation}\label{juhuyuyyg89urtg8rt8g8rt8gh-2}
		\sum_{I \in \cD} \#\pi^{-1}(I) \gtrapprox \#A^2.
	\end{equation}
	In particular, $r_0 \approx \#A^2/\#\cD.$ 
	
	Let $Q_0$ be an $\e$-uniform collection of centers of intervals in $\cD$. Trivially, $\cH^0_{[\delta,\infty)}(Q_0) \ge 1$. Write
	\begin{equation}\label{eq.defoft}
		t := \sup\{0 \leq v \leq 43s/34 :\cH^v_{[\delta,\infty)}(Q_0) \ge 1\}.
	\end{equation}
	In particular, since $\cH^t_{[\delta,\infty)}(Q_0) \ge 1$ we have
	\begin{equation}\label{eq.wearefinsihedd}
		\cns \rho {Q_0} \ge \rho^{-t} \text{ for all } \de \le  \rho \le 1.
	\end{equation}
	This guarantees that $Q_0$ is a $(\de,t)$-set. The rationale behind the choice of $43/34$ is explained in   \eqref{hyreyfruf8ruef8utg98urt8g}. We derive two estimates.
	\subsubsection{First estimate}
	We apply Theorem \ref{thm.proj} to $\cP=A \times A$ and $E=Q_0.$ This tells us that there exists $Q \subset Q_0$ with 
	\begin{equation}
		\label{juhuyuyyg89urtg8rt8g8rt8gh-3}
		\#Q \gtrsim\delta^\e\#Q_0
	\end{equation}
	so that 
	\begin{equation}
		\cont{s+t/2}(A-mA) \gtrsim \delta^{O(\e)} \text{ for all } m \in Q.
	\end{equation} 
	For later use we also note that due to \eqref{juhuyuyyg89urtg8rt8g8rt8gh-}
	, \eqref{juhuyuyyg89urtg8rt8g8rt8gh-2} and 
	\eqref{juhuyuyyg89urtg8rt8g8rt8gh-3} we have
	\begin{equation}
		\label{eq.densearcs}
		\sum_{I \in Q} \#\pi^{-1}(I) \gtrsim \delta^{O(\e)} \#A^2.
	\end{equation}
	For each $m \in Q$ by Lemma \ref{lem.disfrost}  we may find  $\Gamma_m \subset A-mA$ which is a $(\de,s+t/2,\delta^{-O(\e)})$-set. By dropping to a subset  we may assume that  all $\#\Gamma_m$ have the same value, call it $\Gamma$.  Putting these together we see that the set of lines
	\begin{equation}
		\cL:= \{y = mx + c\}_{m \in Q, c \in \Gamma_m}
	\end{equation}
	forms a $$(\delta,t,s+t/2,\#Q\delta^{t-O(\e)},\Gamma\delta^{s+\frac{t}2-O(\e)})-\text{quasi-product set}$$ with $\#\cL=\Gamma\#Q.$ Let $\cT$ be the $O(\de)$-neighbourhoods of the lines in $\cL.$ Thus
	\begin{equation}
		\label{jrhryf7ry78yg7ytg7yt7g8yt7gy}
		\#\cT=\Gamma\#Q.
	\end{equation}

	We now define a shading on $\cT$. Recall that for each $m \in \cD$ the $O(\delta)$-neighbourhood of the line $y = mx$ intersects $\sim r_0$ many $\de$-balls (centered at points) in $A \times A.$ Let $l \in \cL$ and find $a,b \in A$ so that $l$ may be represented by the equation $y = m(x-a) + b.$ Consider the $\de$-separated points
	\begin{eqnarray}
		(a_1,b_1),\ldots ,(a_{r_0},b_{r_0}) \in A \times A 
	\end{eqnarray}
	so that $|b_j/a_j - m| < O(\delta).$ Then for each $P_j=(a_j + a, b_j + b)$  we have
	\begin{equation}
		b_j + b = b_j(a_j + a - a)/a_j + b + O(\delta),
	\end{equation}
	and so the $O(\delta)$-neighbourhood of $l$ contains $\sim r_0$ points $P_j$ from $(A+A) \times (A+A).$ Letting $T$ be the $O(\de)$-neighbourhood of $l$ we define $Y(T)$ to be these $r_0$ points. Thus 
	\begin{equation}
		\label{uhryufreyfryf78ryf78yr}
		\#Y(T) \gtrsim \dep \#A^2/\#Q.
	\end{equation}
	By Lemma \ref{lem.productset}, $Y(T)$ is a $(\delta,s,\delta^{s-O(\e)}\#A)$-KT set. To apply Theorem \ref{thm.product}, we require that $s \leq \min\{3t/2 + s, 2 - 3t/2 -s \}.$ If $3t/2 + s$ is the minimum, then since $t \ge 0$ this is automatically satisfied. Otherwise, we require that 
	\begin{equation}
		s \leq 1 - 3t/4.
	\end{equation}
	Since we assumed that $t \leq 43s/34$, the above inequality holds true when $0 < s \leq 136/265.$	
	
	We apply Theorem \ref{thm.product} to the collection $\cT$ with shading $Y$. We first note that $t\le s+\frac{t}2$, since we operate under the stronger assumption $t\le 43s/34$. Thus, the $\alpha$ in Theorem \ref{thm.product} equals 3 in our case. We have 
	$$\sum_{T\in\cT}\#Y(T)\lesssim  \dem (K_1K_2)^{2/3}K_3^{1/3}(\delta^{-s-\frac{3t}2}\#\cT)^{1/{3}}\# Y(\cT)^{2/3}.$$
	Combining this with \eqref{uhryufreyfryf78ryf78yr} we find 
	$$\frac{(\#A)^2}{\#Q}\#\cT\les \dem(\#Q\Gamma\delta^{s+\frac{3t}{2}})^{2/3}(\#A\delta^s)^{1/3}(\delta^{-s-\frac{3t}2}\#\cT)^{1/{3}}\# Y(\cT)^{2/3}.$$
	When combined with \eqref{jrhryf7ry78yg7ytg7yt7g8yt7gy} this is equivalent with 
	$$\#Q^{-3/2}\de^{-3t/4-s}(\#A)^{5/2}\lesssim \dem\# Y(\cT) \leq \dem \cn {(A+A)\times (A+A)}.$$
	A simple rearrangement tells us that
	$$\cn{A+A}^8 \#Q^6 \gtrsim \delta^{O(\e)}\delta^{-4s-3t}(\#A)^{10}.$$
	Using that $\#A\gtrsim \delta^{-s+O(\e)}$ we find our first estimate\begin{equation}\label{eq.firstestimate}
		\cn{A+A}^8 \#Q^6 \gtrsim \delta^{O(\e)}\delta^{-14s-3t}.
	\end{equation}
	Recall that $Q$ is a $(\delta,t,\delta^{-O(\e)})$-set, with size possibly much larger than $\delta^{-t}$. If the size of $Q$ was $\approx \delta^{-t}$ then \eqref{eq.firstestimate} would imply that
	$$\cn{A+A}^8 \cn{A/A}^3 \gtrsim \delta^{O(\e)}\delta^{-14s},$$
	and thus the very satisfactory 
	$$\max(\cn{A+A},\cn{A/A})\gtrsim \delta^{O(\e)}\delta^{-14s/11}.$$
	\subsubsection{Second estimate}
	
	\begin{theorem} 
		\label{chfuv hu f irtg9irt9h9yh}
		Let $s\le \frac23$. We assume $A$ is $\epsilon$-uniform. Then
		\begin{equation}
			\cont{5s/2}((A+A) \times Q) \gtrsim \delta^{O(\e)}.
		\end{equation}
		\begin{proof}
			For $a,b \in A$ define $l_{a,b}$ to be the line given by the equation
			\begin{equation}
				y = (x-b)/a.
			\end{equation}
			Write 
			\begin{equation}
				\cL_0:= \{l_{a,b} : a,b \in A\}.
			\end{equation}
			Let $\cT_0$ be the $O(\delta)$-neighbourhoods of the lines in $\cL_0.$ Let $T \in \cT_0$ have axial line $l_{a,b}.$ The tube $T$ contains the points
			\begin{equation}
				\{(b+c,c/a): c \in A: c/a \in Q\}\subset (A+A)\times Q.
			\end{equation}
			This set has cardinality the same as that of 
			\begin{equation}
				A_a := \{c \in A : c/a \in Q\}.
			\end{equation}
			We wish that each $T \in \cT_0$ contains  $ \approx \#A$ points. This will be true after a refinement. By \eqref{eq.densearcs} we have
			\begin{equation}
				\sum_{a \in A} \#A_a = \sum_{a,c \in A} 1_{c/a \in Q} \gtrapprox \#A^2.
			\end{equation}
			Therefore for some $A_0 \subset A$ with $\#A_0 \gtrapprox \#A$ we have that $\#A_a \gtrapprox \#A$ for all $a \in A_0.$ Replacing $A_0$ with a $\gtrsim \delta^{\e}$-fraction we may assume that $A_0$ is $\e$-uniform. 
			
			Write
			\begin{equation}
				\cL := \{l_{a,b} : (a,b) \in A_0 \times A\},
			\end{equation}
			and let $\cT$ be the $O(\delta)$-neighbourhoods of the lines in $\cL.$ For a tube $T$ with axial line $l_{a,b}$ define  
			\begin{equation}
				Y(T) := \{(b+c,c/a)_\de : c \in A_a\}\subset (A+A)\times Q.
			\end{equation}
			We aim to apply Theorem \ref{thm.product} to the $\rho$ thickening of the collection $\cT$ - call this $\cT(\rho)$ -  equipped with the shading $Y_\rho$.

			By replacing $Y(T)$ with a $\gtrsim \delta^{\e}$-fraction we may assume that $Y(T)$ is $\epsilon$-uniform.  Since $Y(\cT)\subset (A+A)\times Q$, it will suffice to prove 
			\begin{equation}
				\cont{5s/2}(Y(\cT)) \gtrsim \delta^{O(\e)}.
			\end{equation}
			
			Let $\cU$ be a cover of $Y(\cT)$ with disjoint dyadic squares of side-length at least $\delta.$ By dyadic pigeonholing, for each $T \in \cT$ there is a $j(T)$ so that 
			\begin{equation}
				\#(Y(T) \cap \bigcup_{U \in \cU: \diam(U) = 2^{-j(T)}}U) \gtrapprox \#Y(T) \gtrsim \delta^{O(\e)}\#A.
			\end{equation}
			Again, we may dyadic-pigeonhole some $j$ and a subset $\cT_1 \subset \cT$ with $\#\cT_1 \gtrapprox \#\cT$ so that for all $T \in \cT_1$ we have $j(T) = j.$ Fix $\rho = 2^{-j}.$

			\indent Since $A$ and $A_0$ are  $\epsilon$-uniform, by Proposition \ref{prop.properties},  both $A_\rho$ and $(A_0)_\rho$ are $(\rho,s,\delta^{-O(\e)})$-sets with sizes 
			\begin{equation}
				\label{hcyufrueg9rtugrt-9}
				\cns\rho A, \cns{\rho}{A_0}\approx N\gtrsim \de^{O(\e)}\rho^{-s}.
			\end{equation}
			Therefore $\cT(\rho)$ is a $$(\rho,s,s,\dem N_\rho(A_0)\rho^s, \dem N_{\rho}(A)\rho^s)\text{-quasi-product set}$$ with $$\#\cT(\rho)=N_\rho(A_0)N_{\rho}(A).$$ We may define a  shading $Y_\rho$ on each $S \in \cT(\rho)$ by taking the union of the $U \in \cU$ with $\diam(U) = \rho$ which intersect one of the $T \subset S.$ We have $\#Y_\rho(S) \gtrsim \delta^{O(\e)}N_\rho(A)$. Moreover, $Y_\rho(S)$ is a $(\rho,s,K_3)$-KT set with $K_3\lesssim \dem N_\rho(A)\rho^s$. 
			
			Now apply Theorem \ref{thm.product} to $\cT_\rho$ and the shading $Y_\rho$. Since $s\le 2/3$ we have that $s\le \min(2s,2-2s)$. Also, the parameter $\alpha$ equals 3 in this case. Thus 
			$$\sum_{S\in\cT(\rho)} \#Y_\rho(S)\lesssim \dem K_3^{\frac13}(K_1K_2)^{2/3} (\rho^{-2s}\#\cT(\rho))^{1/{3}}(\# Y_\rho(\cT(\rho)))^{\frac23}.$$
			Using the parameter $N$ from \eqref{hcyufrueg9rtugrt-9}, this may be rewritten as
			$$N^3\lesssim \dem (N\rho^s)^{1/3}(N^2\rho^{2s})^{2/3}(\rho^{-2s}N^2)^{1/3}(\# Y_\rho(\cT(\rho)))^{\frac23},$$
			or $$\# Y_\rho(\cT(\rho))\gtrsim \de^{O(\e)}\frac{N}{\rho^{3s/2}}.$$
			Using our lower bound on $N$ leads to 
			$\# Y_\rho(\cT(\rho))\gtrsim \de^{O(\e)}\rho^{-5s/2}$. Note that by definition $Y_\rho(\cT(\rho))\subset \{U\in\cU:\;\diam(U)=\rho\}$, and thus
			$$\sum_{U\in\cU}\diam(U)^{5s/2}\ges \delta^{O(\e)}.$$ 
			Since $\cU$ was an arbitrary cover, it follows  that
			\begin{equation}
				\cont{5s/2}(Y(\cT)) \gtrsim \delta^{O(\e)},
			\end{equation}
			as required.
		\end{proof}
	\end{theorem}
	\begin{corollary}
		\label{ huhv yug8ytg7yg7yg7yt8gy7gy}Let $s\le \frac23$. Assume the strict inequality $t < 43s/34.$ Assume $A$ is $\e$-uniform. 
		Let $u$ be such that 
		\begin{equation}\label{eq.upperbox}
			\cns{\rho}{A+A} < \rho^{-u} \text{ for all } \delta \leq \rho \leq \delta^{O(\e)}.
		\end{equation}
		Then \begin{equation}
			\label{c nh fdvhuivju rtgtig}
			\cont{5s/2 - u}(Q) \gtrsim \de^{O(\e)}.
		\end{equation}
		\begin{proof}
			Due to the maximality of $t$ in \eqref{eq.defoft}, 
			for any $\tau > t$ we may find a cover $\cU$ of $Q$ (recall $Q$ is a subset of $Q_0$) by dyadic intervals of length at least $\delta$ so that
			\begin{equation}
				\label{iojfurh fuygytgytg8yrt8g9987}
				\sum_{U \in \cU} \diam(U)^\tau < 1.
			\end{equation}
			For each $U \in \cU$, by \eqref{eq.upperbox} we may cover $A+A$ by a collection $\cI_U$ consisting of 
			$N_U<\diam(U)^{-u}$ many intervals $I$ of length $\rho=\diam(U).$ This leads to a cover  of $(A+A) \times Q$ consisting of all squares in $V\in \cV:=\cup_{U\in\cU}\cup_{I\in\cI_U}I\times U$. Due to \eqref{iojfurh fuygytgytg8yrt8g9987} we have
			\begin{equation}
				\sum_{V\in \cV} \diam (V)^{\tau+u} = \sum_{U \in \cU} \diam(U)^{\tau +u}\diam(U)^{-u} < 1. 
			\end{equation}
			By Theorem \ref{chfuv hu f irtg9irt9h9yh} we therefore have $\tau + u \geq 5s/2,$ and so $\tau \geq 5s/2 - u.$ Since $\tau > t$ is arbitrary, it follows that $t \geq 5s/2 - u$ and by the definition of $t$ we must have
			$$\cont{5s/2 - u}(Q_0) \ge 1.$$
			This in turn implies \eqref{c nh fdvhuivju rtgtig}.
		\end{proof}
		
	\end{corollary}
	We are now ready to prove 
	Theorem \ref{thm.dis}, which we recall below.
	\begin{theorem}
		Let $0 < s \leq 136/265.$ There exists $C,\delta_0, \epsilon> 0$ so that the following holds for all $0 < \delta < \delta_0.$
		
		\indent Let $A \subset [1/2,1]$ be a $(\delta, s, \de^{-\e})$-set. Then one of the following must hold.
		\begin{equation}\label{eq.a+a}
			\cns{\rho}{A+A} \ge \rho^{-43s/34} \text{ for some } \delta \le \rho \lesssim  \delta^{C\e},
		\end{equation}
		\begin{equation}\label{eq.haus}
			\cn{A/A} > \delta^{-43s/34+C\e}.
		\end{equation}
	\end{theorem}
	\begin{proof} 
		We may certainly assume $A$ is $\e$-uniform.
		We may assume that
		$$\cns{\rho}{A+A} < \rho^{-43s/34} \text{ for each } \delta \le \rho \lesssim  \delta^{C\e},$$
		otherwise \eqref{eq.a+a} holds.
		
		We may also assume that $t<43s/34$, as otherwise $t=43s/34$, and this together with \eqref{eq.wearefinsihedd} for $\rho=\delta$
		\begin{equation}
			\cn{A/A} \geq \cn{Q_0} \geq \de^{-43t/34}
		\end{equation} 
		would imply \eqref{eq.haus}.

		Note that $s<2/3$. Then using the conclusion \eqref{c nh fdvhuivju rtgtig} of Corollary \ref{ huhv yug8ytg7yg7yg7yt8gy7gy} and the maximality of $t$ it follows that $t\ge \frac{5s}{2}-\frac{43s}{34}$.
		
		We recall \eqref{eq.firstestimate} 
		\begin{equation}
			\label{h gdygyuryeryuefyuryyg}
			\cn{A+A}^8 \# Q ^6 \gtrsim \delta^{O(\e)}\delta^{-14s-3t},
		\end{equation}
		which when combined with $N_\delta(A+A)\le \delta^{-43s/34}$ leads to 
		$$\#Q^6\gtrsim \delta^{O(\e)}\delta^{-14s-3t+\frac{43\times 8s}{34}}.$$Since $t\ge \frac{5s}{2}-\frac{43s}{34}$ this proves that
		$\cn{A/A}\ge \#Q \gtrsim \delta^{O(\e)}\delta^{-43s/34}.$ 
		The exponent $p=43/34$ was chosen for the following reason. It is the smallest that makes the inequality \eqref{h gdygyuryeryuefyuryyg} true, which after applying $\log_{1/\delta}$ reads
		\begin{equation}
			\label{hyreyfruf8ruef8utg98urt8g}
			8ps+6ps\ge 14s+3(\frac52s-ps).
		\end{equation}
	\end{proof}
	\subsection{Proof of Theorem \ref{thm.sumprod}}
	We first recall what we are about to prove.
	\begin{theorem}
		Let $0 < s \leq 136/265.$  For all $A \subset \R$ with $\dimh A = s$ we have
		\begin{equation}
			\max\{\ubox(A+A),\lbox(A/A)\} \geq 43s/34.
		\end{equation}
	\end{theorem}
	\begin{proof}
		Let $\e > 0.$ There exists $n_0$ so that for all $n \geq n_0$, by Lemma \ref{lem.disfrost} we may find $A_n \subset A$ which is a $(2^{-n},s, C)$-set where $C \sim    1/\cH^{s}_\infty(A).$ If for all but finitely many $n$ we have $\cns{2^{-n}}{A_n/A_n} \gtrsim 2^{n(43s/34-O(\e))}$ then clearly 
		\begin{equation}
			\lbox(A/A) \geq 43s/34 - O(\epsilon).
		\end{equation}
		Otherwise, by Theorem \ref{thm.dis} for infinitely many $n$ we may find a scale $2^{-n} \leq \rho_n \leq 2^{-O(\e)n}$ with
		\begin{equation}
			\cns{\rho_n}{A_n + A_n} \geq \rho_n^{-43s/11}.
		\end{equation}
		Clearly then $\ubox(A+A) \geq 43s/34.$ Since $\e$ was arbitrary the result follows. 
	\end{proof}
	\subsection{Proof of Theorem \ref{thm.fewsums}}\label{sect.fewsums}
	We restate Theorem \ref{thm.fewsums} with the assumption on $\#A$ dropped.
	\begin{prop}\label{prop.fewsum}
		Let $0 \leq s \leq 2/3.$ There exists $C,\delta_0, \epsilon> 0$ so that the following holds for all $0 < \delta < \delta_0.$ 
		
		Let $A \subset [1/2,1]$ be a $\delta$-separated 
		$(\delta,s,\de^{-\e})$-set. If $s \leq 1/2,$ then
		\begin{equation}
			N_\delta(A+A)^6N_\delta(A/A)> \de^{C\e}\# A^8.
		\end{equation}
		If $s>1/2,$ then
		\begin{equation}
			N_\delta(A+A)^6N_\delta(A/A) > \de^{C\e}\de^{-2}\# A^4.
		\end{equation}
	\end{prop}
	This will follow from a similar statement with an additional separation condition on $A/A.$ Set $\sigma := \min\{2s,2-2s\}$, which will be used throughout this section.
	\begin{theorem}\label{thm.fewsumweak}
		Let $0 < s < 1.$ Let $A \subset [1/2,1]$ be a $\delta$-separated
		$(\delta,s,C)$-KT set. Suppose further that $A/A$ is a $(\de,\sigma)$-KT set. Then
		\begin{equation}
			N_\delta(A+A)^6N_\delta(A/A) \gtrapprox C^{-4}\de^{2s}\#A^{10}.
		\end{equation}
	\end{theorem}
	With Theorem \ref{thm.fewsumweak} in hand, we move to the proof of Proposition \ref{prop.fewsum}. We first show that for uniform sets, additive structure is retained at courser scales. 
	\begin{lemma}\label{lem.sumsetcontrol}
		Assume $A$ is $\e$-uniform. For all $\de < \rho < 1$ we have
		\begin{equation}
			\frac{\cns{\rho}{A+A}}{\cns{\rho}{A}} \lesssim_\e \frac{\cn{A+A}}{\# A}.
		\end{equation}
		\begin{proof}
			Due to the uniformity of $A$, each $I\in \cD_{\rho/2}(A)$ contains $M\sim_\e \#A/N_{\rho/2}(A)\sim \#A/N_\rho(A)$ points in $A$. Consider the collection of all distinct intervals of the form $I_1+I_2$ ($I_i\in\cD_{\rho/2}(A)$). Each point appears in at most three intervals. They have length $\rho$, cover $A+A$, and each interval intersects $A+A$. The number of these intervals is thus $\ge N_\rho(A+A)$. And each interval $I_1+I_2$ contains at least $\sim M$ many $\delta$-separated points in $A+A$, for example $a+a'$ for fixed $a\in I_1$ and $a'\in I_2\cap A$. It follows that          
			\begin{equation}
				\frac{\cn{A+A}}{\# A} \gtrsim \frac{\cns{\rho}{A+A}M}{\# A} \sim \frac{\cns{\rho}{A+A}}{\cns{\rho}{A}}.
			\end{equation}
		\end{proof}
	\end{lemma}
	Let $\e_0=\e^{2\e^{-1}}\ll \e$.
	Next, we show that for any exponent $0< t \leq 1,$ and for any finite $\e_0$-uniform $\de$-separated $B \subset \R$ one may always find a courser scale $\rho$ so that $B$ will be a $(\rho,t)$-KT set, with `large' branching thereafter. 
	\begin{lemma}\label{lem.indscale} Assume $B\subset\R$ is $\e_0$-uniform and let $0< t \leq 1$.
		There exists $\delta \le  \rho \le 1$ so that
		\begin{enumerate}
			\item $B$ is a $(\rho,t)$-KT set.
			\item for each $I \in \cD_{\rho}(B)$ we have  $\cn{B \cap I} \gtrsim (\rho/\delta)^{t - O(\e)}.$
		\end{enumerate}
		\begin{proof} 
			We apply the multi-scale decomposition \cite[Theorem 2.13]{demwangszem} to find an integer $L,$ a partition 
			\begin{equation}
				\de := \rho_1 < \rho_2 < \ldots < \rho_{L+1} = 1,
			\end{equation} 
			and a sequence 
			\begin{equation}
				1 \geq t_1 > t_2 > \ldots > t_L \geq 0,
			\end{equation}
			so that  we have
			\begin{equation}
				\label{lckj fdbhgbhugb}
				(\frac{\rho_{j+1}}{\rho_{j}})^{t_j- 3\e}\le \frac{\cns{\rho_{j}}{B}}{\cns{\rho_{j+1}}{B}} \leq (\frac{\rho_{j+1}}{\rho_{j}})^{t_j + 3\e},
			\end{equation}
			and for all $I \in \cD_{\rho_j}(B)$ we have
			\begin{equation}\label{eq.frostmanpieces}
				B \cap I \text{ rescaled by $ \rho_j^{-1}$ is a }(\rho_{j-1}/\rho_j, t_j, O((\rho_j/\rho_{j-1})^{3\e}))\text{-set}.
			\end{equation}
			If $t_L > t - 10\e,$ then multiplying the lower bounds in \eqref{lckj fdbhgbhugb} we find
			$\cn {B} \geq \de^{-t+O(\e)},$
			giving us the statement with $\rho = 1.$ 
			If $t_1\le t-10\e$, we may take $\rho=\delta$, by multiplying the upper  bounds in \eqref{lckj fdbhgbhugb}.
			In all other cases
			let $j$ be the smallest so that $t_j \leq t-10\e$ and select $\rho := \rho_j.$ Note that $j\ge 2$.

			We first verify condition (2). For $I \in \cD_\rho(B),$ we have, using our choice of $j,$ and \eqref{eq.frostmanpieces}, 
			\begin{align}
				\cn{B \cap I} &\sim \frac{\cn{B}}{\cns{\rho}{B}}\\
				&= \prod_{1\le i\le  j-1} \frac{\cns{\rho_i}{B}}{\cns{\rho_{i+1}}{B}} \\
				& \geq \prod_{1\le  i\le j-1}	(\frac{\rho_{i+1}}{\rho_{i}})^{t_i - 3\e} \\
				& > \prod_{1\le i\le j-1}	(\frac{\rho_{i+1}}{\rho_{i}})^{t - O(\e)} \\
				&= (\frac\rho\de)^{t-O(\e)}.
			\end{align}
			We next verify (1). Let $r > \rho$ and select $j \leq l \leq L$ so that
			$
			\rho_{l} \leq r \leq \rho_{l+1}. 
			$
			By similar reasoning, for each interval $J$ of length $r$
			\begin{align*}
				\cns{\rho}{B \cap J} &=\#\{I\in\cD_\rho(B):\;I\subset J\}\\&=\prod_{i=j}^{l-1}\frac{N_{\rho_i}(B)}{N_{\rho_{i+1}(B)}}\#\{I\in\cD_{\rho_l}(B):\;I\subset J\}\\&\lesssim
				\prod_{i=j}^{l-1}\frac{N_{\rho_i}(B)}{N_{\rho_{i+1}(B)}}
				(\frac{r}{\rho_{l+1}})^{t-7\e}\frac{N_{\rho_l}(B)}{N_{\rho_{l+1}}(B)}(\frac{\rho_{l+1}}{\rho_l})^{3\e}, \text{ by }\eqref{eq.frostmanpieces}
				\\&\lesssim\prod_{i=j}^{l}(\frac{\rho_{i+1}}{\rho_i})^{t-7\e}(\frac{r}{\rho_{l+1}})^{t-7\e}(\frac{\rho_{l+1}}{\rho_l})^{3\e}\\&\le\prod_{i=j}^{l}(\frac{\rho_{i+1}}{\rho_i})^{t}(\frac{r}{\rho_{l+1}})^{t}\\
				&=  (\frac r\rho)^{t}.
			\end{align*}
			Therefore $B$ is a $(\rho,t)$-set.
		\end{proof}
	\end{lemma}
	
	With these lemmata, and Theorem \ref{thm.fewsumweak} in hand, we can now  easily complete the proof of Proposition \ref{prop.fewsum}
	
	\begin{proof}[Proof of Proposition \ref{prop.fewsum}]
		We replace $A$ and $A/A$ with $\e_0$-uniform subsets retaining a $\delta^\e$-fraction. 
		Apply Lemma \ref{lem.indscale} to $B=A/A$ with exponent $t=\sigma$ to acquire a scale $\de \leq \rho \leq 1.$ By Proposition \ref{prop.properties}, $A$ is a $(\rho,s, \de^{-O(\e)}\cns{\rho}{A} \rho^s)$-KT set. Apply Theorem \ref{thm.fewsumweak} at this scale $\rho,$ where we find that
		\begin{equation}
			\cns{\rho}{A+A}^6\cns{\rho}{A/A} \gtrsim \de^{O(\e)}\rho^{-2s}\cns{\rho}{A}^6.
		\end{equation}
		Using item (2) from Lemma \ref{lem.indscale}, and the uniformity of $A/A,$ we find that
		\begin{equation}
			\cns{\rho}{A+A}^6\cn{A/A} \gtrsim \de^{O(\e)}\rho^{-2s}(\rho/\de)^{\sigma}\cns{\rho}{A}^6.
		\end{equation}
		By Lemma \ref{lem.sumsetcontrol}, we have
		\begin{equation}
			\cns{\de}{A+A}^6\cn{A/A} \gtrsim \de^{O(\e)}\rho^{-2s}(\rho/\de)^{\sigma}\# A^6.
		\end{equation}
		If $s \leq 1/2,$ then we have
		\begin{equation}
			\cns{\de}{A+A}^6\cn{A/A} \gtrsim \de^{O(\e)}\de^{-2s}\# A^6.
		\end{equation}
		Otherwise, we are led to
		\begin{align}
			\cns{\de}{A+A}^6\cn{A/A} \gtrsim \de^{O(\e)}\rho^{2-4s}\de^{2s-2}\# A^6 \gtrsim \de^{O(\e)}\de^{2s-2}\# A^6.
		\end{align}
	\end{proof}
	
	We dedicate the rest of this section to proving Theorem \ref{thm.fewsumweak}. Assume $A$ is a $\delta$-separated
	$(\delta,s,C)$-KT set. We may  replace $A$ by an $\e$-uniform subset retaining a $\gtrsim \de^\e$-fraction.
	
	Let $\pi:\R^2 \rightarrow \R$ be the map $(x,y) \mapsto x+y.$ Since 
	\begin{equation}
		\sum_{I \in \cD_\de(A+A)} \# (\pi^{-1}(I) \cap A^2) = \# A^2,
	\end{equation} 
	by dyadic-pigeonholing, we may find a dyadic integer $N,$ and a subset $\cD \subset \cD_\de(A+A),$ so that 
	for all $I \in \cD$ we have 
	\begin{equation}
		\#(\pi^{-1}(I) \cap A^2) \sim N,
	\end{equation}
	and 
	\begin{equation}
		\sum_{I \in \cD} \# (\pi^{-1}(I) \cap A^2) \approx \# A^2.
	\end{equation}
	Since $A$ is $\delta$-separated it follows that $N\lesssim \#A$.
	Let $S_0$ be a maximal $\delta$-separated set in  $(\cup_{ I\in \cD}I) \cap (A+A),$ and let $S \subset S_0$ be $\e$-uniform with $\# S \gtrsim \de^{\e}\#S_0.$ In particular,
	\begin{equation}\label{eq.large}
		\sum_{I \in \cD_\de(S)}\# (\pi^{-1}(I) \cap A^2) \gtrapprox \de^\e\# A^2.
	\end{equation}  
	Since $N\lesssim \#A$, we must have $\#S\ges \delta^{\e}\#A$.
	Set $$G := (\bigcup_{I \in \cD_\de(S)}\pi^{-1}(I)) \cap A^2 \subset A \times A,$$
	which by \eqref{eq.large} satisfies 
	\begin{equation}\label{eq.large2}
		\# G \gtrapprox \de^{\e}\# A^2.
	\end{equation}
	Let us see how this implies that 
	\begin{equation}
		\label{d jcfvhuyt7y0gy098gu8}
		N_r(S)\ges \delta^{\e}N_r(A), \;\;\forall r\ge \delta.
	\end{equation}
	Fix $r$ and let $T=\#A/N_r(A)$, so that $\#(A\cap I)\sim T$ for each $J\in\cD_r(A)$. Then 
	$$\#G\lesssim T\sum_{a\in A}\#\{J\in \cD_r(A):\;\exists a_J\in J\text{ so that }(a,a_J)\in G\}.$$
	It follows that there is $a\in A$ and $\gtrsim \frac{\#G}{T\#A}$ many intervals $J\in \cD_r(A)$ such that $(a,a_J)\in G$. This means that $a+a_J\in S$, and since they are $r$-separated, it follows that $N_r(S)\gtrsim \frac{\# G}{T\#A}$. Now \eqref{d jcfvhuyt7y0gy098gu8} follows from \eqref{eq.large2}.

	Let $K > 1$ be the doubling factor of $A$: 
	\begin{equation}
		K := \frac {\cov{A+A}} {\# A}.
	\end{equation}
	
	Recall that $A$ is a $(\delta,s,C)$-KT set.
	\begin{lemma}
		The set $S$ is a  $(\delta,s,L)$-KT set with $L\approx CK\delta^{-\epsilon}$.
		\begin{proof}
			Let $r \ge \de$ and $I \in \cD_r(S).$ By first using the uniformity of $S$, then \eqref{d jcfvhuyt7y0gy098gu8} and finally the fact that $A$ is $(\delta,s,C)$-KT,  we have
			\begin{equation}
				\#(S\cap I)  \sim \frac{\#S}{\cns{r}{S}} \lessapprox  \delta^{-\epsilon}\frac{K\# A}{\cns{r}{A}} \leq L(r/\de)^s.
			\end{equation}
		\end{proof}
	\end{lemma}

	Now set $\cP := (A \cup S) \times (A \cup S),$ which is a $(\de,s,s,L,L)$-quasi-product set, with $\#\cP\les \delta^{-O(\e)}\#S^2$.  Let $\cT$ be the collection of $\de$-tubes with the slope  in a maximal $\delta$-separated set $D$ of $A/A$, and which contain at least one point in $\cP.$ Let $r \geq 1,$ and call $\cT_r(\cP)$ the collection of those $T\in\cT$ containing $\sim r$ many $p\in\cP$. We apply Corollary  \ref{nononecase} in its dual form to obtain
	\begin{equation}\label{eq.tuberesult}
		\#\cT_{r}(\cP) \lessapprox L^4 \delta^{-2s}\#\cP /r^3\les \delta^{-O(\e)} L^4 \delta^{-2s}\#S^2 /r^3.
	\end{equation}

	We wish to count the number of triples $(p,q,s) \in \cP$ which all lie in some $T \in \cT.$
	The number of triples that lie in a single $T \in \cT_r(\cP)$ is  $\sim r^3.$ By \eqref{eq.tuberesult} the number of triples which lie in some $T\in \cT_r(\cP)$ is at most
	\begin{eqnarray}
		\les \delta^{-O(\e)} L^4\de^{-2s}\#S^2.
	\end{eqnarray} 
	Summing over each dyadic $r$ tells us the total number of such triples is at most
	\begin{equation}\label{eq.triplesupper}
		\les\de^{-O(\e)}L^4\de^{-2s}\#S^2.
	\end{equation} 
	\indent We now wish to find a lower bound for the number of distinct triples of the form
	\begin{eqnarray}
		(e,f) \in A\times A, (e+a,f+b), (e+c,f+d) \in S\times S,
	\end{eqnarray}
	with $|b/a - d/c| \leq \delta.$ That is, we may find $T \in \cT$ so that the triple above lies in $T.$ For $(e,f) \in A \times A$ define
	\begin{eqnarray}
		G_{e,f}:= \{(a,b) \in A \times A: (e+a,f+b) \in S^2\}.
	\end{eqnarray}
	Since, trivially
	\begin{equation}
		\#G_{e,f} \sim \sum_{q \in D}\#\{(a,b) \in G_{e,f} : |b/a - q| \leq \de\},
	\end{equation}
	by Cauchy--Schwarz we have
	\begin{equation}
		\#\{(a,b,c,d) \in (G_{e,f})^2: |b/a - d/c| \lesssim \de\} \gtrsim (\#G_{e,f})^2/\#D.
	\end{equation}
	Therefore for each $(e,f) \in A\times A$ we may find $\gtrsim (\#G_{e,f})^2/\#D$ pairs in $\cP$ which along with $(e,f)$ form a triple as above. We have therefore found at least
	\begin{equation}\label{eq.triples2}
		\gtrsim (\#D)^{-1}\sum_{e,f \in A}(\#G_{e,f})^2
	\end{equation}
	such triples.
	Since
	\begin{equation}
		\sum_{(e,f) \in A^2} \# G_{e,f} = \#\{(e,f,a,b) \in A^4: (e+a,f+b) \in S^2\} = \# G^2,
	\end{equation}
	by Cauchy--Schwarz and \eqref{eq.large2} we have
	\begin{equation}
		\sum_{(e,f) \in A^2} (\# G_{e,f})^2 \geq \#G^4/\# A^2 \gtrapprox \de^{O(\e)}\# A^6.
	\end{equation}
	By \eqref{eq.triples2}, we have found at least
	\begin{equation}
		\gtrapprox \de^{O(\e)}\# A ^6/ \#D 
	\end{equation}
	such triples. Combining this with \eqref{eq.triplesupper}, we have
	\begin{equation}
		C^4K^6 \delta^{-2s}\# A ^2  \gtrapprox L^4 \de^{-2s}\#S ^2 \gtrapprox \de^{O(\e)}\# A ^6/ \#D .
	\end{equation}
	Using the fact that $K$ was chosen so that 
	\begin{equation}
		\cn{A+A}/\#A = K,
	\end{equation}
	we arrive at 
	\begin{equation}
		\cn{A+A}^6\cn{A/A} \gtrapprox C^{-4}\de^{O(\e)}\de^{2s}\#A^{10}.
	\end{equation}
	Since $\e >0$ was arbitrary, the result follows. 
	
	\bibliographystyle{alpha}
	\bibliography{references}

@article {guthincidence,
	AUTHOR = {Guth, Larry and Solomon, Noam and Wang, Hong},
	TITLE = {Incidence estimates for well spaced tubes},
	JOURNAL = {Geom. Funct. Anal.},
	FJOURNAL = {Geometric and Functional Analysis},
	VOLUME = {29},
	YEAR = {2019},
	NUMBER = {6},
	PAGES = {1844--1863},
	ISSN = {1016-443X,1420-8970},
	MRCLASS = {52C10 (42B25)},
	MRNUMBER = {4034922},
	MRREVIEWER = {Xiumin\ Du},
	DOI = {10.1007/s00039-019-00519-y},
	URL = {https://doi.org/10.1007/s00039-019-00519-y},
}

@article {mator,
	title={Discretised sum-product theorems by {S}hannon-type inequalities},
	author={M\'ath\'e, Andras and O'Regan, William},
	journal={arXiv preprint arXiv:2306.02943},
	year={2023}
}

@article {bou,
	AUTHOR = {Bourgain, J.},
	TITLE = {The discretized sum-product and projection theorems},
	JOURNAL = {J. Anal. Math.},
	FJOURNAL = {Journal d'Analyse Math\'{e}matique},
	VOLUME = {112},
	YEAR = {2010},
	PAGES = {193--236},
	ISSN = {0021-7670,1565-8538},
	MRCLASS = {11B30},
	MRNUMBER = {2763000},
	MRREVIEWER = {Serge\u{\i}\ V.\ Konyagin},
	DOI = {10.1007/s11854-010-0028-x},
}

@article {bou03,
	AUTHOR = {Bourgain, J.},
	TITLE = {On the {E}rd\"os-{V}olkmann and {K}atz-{T}ao ring conjectures},
	JOURNAL = {Geom. Funct. Anal.},
	FJOURNAL = {Geometric and Functional Analysis},
	VOLUME = {13},
	YEAR = {2003},
	NUMBER = {2},
	PAGES = {334--365},
	ISSN = {1016-443X},
	MRCLASS = {11K55 (28A80)},
	MRNUMBER = {1982147},
	MRREVIEWER = {Ben Joseph Green},
	DOI = {10.1007/s000390300008},
	URL = {https://doi.org/10.1007/s000390300008},
}

@article {che,
	AUTHOR = {Chen, Changhao},
	TITLE = {Discretized sum-product for large sets},
	JOURNAL = {Mosc. J. Comb. Number Theory},
	FJOURNAL = {Moscow Journal of Combinatorics and Number Theory},
	VOLUME = {9},
	YEAR = {2020},
	NUMBER = {1},
	PAGES = {17--27},
	ISSN = {2220-5438},
	MRCLASS = {11B30 (05D05)},
	MRNUMBER = {4066556},
	MRREVIEWER = {P\'{e}ter P\'{a}l Pach},
	DOI = {10.2140/moscow.2020.9.17},
	URL = {https://doi.org/10.2140/moscow.2020.9.17},
}

@article {erd,
	AUTHOR = {Erd\H{o}s, Paul and Volkmann, Bodo},
	TITLE = {Additive {G}ruppen mit vorgegebener {H}ausdorffscher
		{D}imension},
	JOURNAL = {J. Reine Angew. Math.},
	FJOURNAL = {Journal f\"{u}r die Reine und Angewandte Mathematik. [Crelle's
		Journal]},
	VOLUME = {221},
	YEAR = {1966},
	PAGES = {203--208},
	ISSN = {0075-4102},
	MRCLASS = {22.10 (20.00)},
	MRNUMBER = {186782},
	MRREVIEWER = {S. J. Taylor},
}

@incollection {erdsz,
	AUTHOR = {Erd\H{o}s, P. and Szemer\'{e}di, E.},
	TITLE = {On sums and products of integers},
	BOOKTITLE = {Studies in pure mathematics},
	PAGES = {213--218},
	PUBLISHER = {Birkh\"{a}user, Basel},
	YEAR = {1983},
	MRCLASS = {11B05},
	MRNUMBER = {820223},
	MRREVIEWER = {Anne Ludington Young},
}

@article {gut,
	AUTHOR = {Guth, Larry and Katz, Nets Hawk and Zahl, Joshua},
	TITLE = {On the discretized sum-product problem},
	JOURNAL = {Int. Math. Res. Not. IMRN},
	FJOURNAL = {International Mathematics Research Notices. IMRN},
	YEAR = {2021},
	NUMBER = {13},
	PAGES = {9769--9785},
	ISSN = {1073-7928},
	MRCLASS = {11B30 (05D10 28A99)},
	MRNUMBER = {4283564},
	MRREVIEWER = {Hanbin Zhang},
	DOI = {10.1093/imrn/rnz360},
	URL = {https://doi.org/10.1093/imrn/rnz360},
}

@article {furen,
	AUTHOR = {Fu, Yuqiu and Ren, Kevin},
	TITLE = {Incidence estimates for {$\alpha$}-dimensional tubes and
		{$\beta$}-dimensional balls in {$\Bbb R^2$}},
	JOURNAL = {J. Fractal Geom.},
	FJOURNAL = {Journal of Fractal Geometry. Mathematics of Fractals and
		Related Topics},
	VOLUME = {11},
	YEAR = {2024},
	NUMBER = {1-2},
	PAGES = {1--30},
	ISSN = {2308-1309,2308-1317},
	MRCLASS = {28A80 (28A75 42B10)},
	MRNUMBER = {4751206},
	DOI = {10.4171/jfg/143},
	URL = {https://doi.org/10.4171/jfg/143},
}

@article {kat,
	AUTHOR = {Katz, Nets Hawk and Tao, Terence},
	TITLE = {Some connections between {F}alconer's distance set conjecture
		and sets of {F}urstenburg type},
	JOURNAL = {New York J. Math.},
	FJOURNAL = {New York Journal of Mathematics},
	VOLUME = {7},
	YEAR = {2001},
	PAGES = {149--187},
	MRCLASS = {28A80 (28A75 28A78)},
	MRNUMBER = {1856956},
	MRREVIEWER = {Miguel Angel Mart\'{\i}n},
	URL = {http://nyjm.albany.edu:8000/j/2001/7_149.html},
}

@article {orpshabc,
	title={Projections, {F}urstenberg sets, and the ABC sum-product problem},
	author={Orponen, Tuomas and Shmerkin, Pablo},
	journal={arXiv preprint arXiv:2301.10199},
	year={2023}
}

@article {renwang,
	title={Furstenberg sets estimate in the plane},
	author={Kevin Ren and Hong Wang},
	journal={arXiv preprint arXiv:2308.08819},
	year={2023}
}

@article{elekes1997number,
title={On the number of sums and products},
author={Elekes, Gy{\"o}rgy},
journal={Acta Arithmetica},
volume={81},
number={4},
pages={365--367},
year={1997}
}

@article {bougam,
AUTHOR = {Bourgain, Jean and Gamburd, Alex},
TITLE = {On the spectral gap for finitely-generated subgroups of {$\rm
SU(2)$}},
JOURNAL = {Invent. Math.},
FJOURNAL = {Inventiones Mathematicae},
VOLUME = {171},
YEAR = {2008},
NUMBER = {1},
PAGES = {83--121},
ISSN = {0020-9910,1432-1297},
MRCLASS = {22E30 (11B75 43A75)},
MRNUMBER = {2358056},
MRREVIEWER = {Ben\ Joseph\ Green},
DOI = {10.1007/s00222-007-0072-z},
URL = {https://0-doi-org.pugwash.lib.warwick.ac.uk/10.1007/s00222-007-0072-z},
}

@article {fassorp,
AUTHOR = {F\"assler, Katrin and Orponen, Tuomas},
TITLE = {On restricted families of projections in {$\Bbb R^3$}},
JOURNAL = {Proc. Lond. Math. Soc. (3)},
FJOURNAL = {Proceedings of the London Mathematical Society. Third Series},
VOLUME = {109},
YEAR = {2014},
NUMBER = {2},
PAGES = {353--381},
ISSN = {0024-6115,1460-244X},
MRCLASS = {28A80 (28A78 37C45)},
MRNUMBER = {3254928},
MRREVIEWER = {Jonathan\ MacDonald\ Fraser},
DOI = {10.1112/plms/pdu004},
URL = {https://doi.org/10.1112/plms/pdu004},
}

@article {demwangszem,
AUTHOR = {Demeter, Ciprian and Wang, Hong},
TITLE = {Szemer\'edi-{T}rotter bounds for tubes and applications},
JOURNAL = {Ars Inven. Anal.},
FJOURNAL = {Ars Inveniendi Analytica},
YEAR = {2025},
PAGES = {Paper No. 1, 46},
ISSN = {2769-8505},
MRCLASS = {42B10},
MRNUMBER = {4869897},
}

@misc{wangwufurst,
title = {Two-ends Furstenberg estimates in the plane},
author={Wang, Hong and Wu, Shukun},
year={2025+} }

@article {orpyi,
    AUTHOR = {Orponen, Tuomas and Yi, Guangzeng},
     TITLE = {Large cliques in extremal incidence configurations},
   JOURNAL = {Rev. Mat. Iberoam.},
  FJOURNAL = {Revista Matem\'{a}tica Iberoamericana},
    VOLUME = {41},
      YEAR = {2025},
    NUMBER = {4},
     PAGES = {1431--1460},
      ISSN = {0213-2230},
   MRCLASS = {28A80 (05B99 05D99 51A20)},
  MRNUMBER = {4912925},
       DOI = {10.4171/rmi/1552},
       URL = {https://doi.org/10.4171/rmi/1552},
}

@article {BD,
    AUTHOR = {Bourgain, Jean and Demeter, Ciprian},
     TITLE = {The proof of the {$l^2$} decoupling conjecture},
   JOURNAL = {Ann. of Math. (2)},
  FJOURNAL = {Annals of Mathematics. Second Series},
    VOLUME = {182},
      YEAR = {2015},
    NUMBER = {1},
     PAGES = {351--389},
      ISSN = {0003-486X,1939-8980},
   MRCLASS = {42B37 (11E76 46E30 53C40)},
  MRNUMBER = {3374964},
MRREVIEWER = {G.\ V.\ Rozenblum},
       DOI = {10.4007/annals.2015.182.1.9},
       URL = {https://doi.org/10.4007/annals.2015.182.1.9},
}

@article {DD,
    AUTHOR = {Dasu, Shival and Demeter, Ciprian},
     TITLE = {Fourier decay for curved {F}rostman measures},
   JOURNAL = {Proc. Amer. Math. Soc.},
  FJOURNAL = {Proceedings of the American Mathematical Society},
    VOLUME = {152},
      YEAR = {2024},
    NUMBER = {1},
     PAGES = {267--280},
      ISSN = {0002-9939,1088-6826},
   MRCLASS = {42B10},
  MRNUMBER = {4661079},
MRREVIEWER = {Terence\ L. J. Harris},
       DOI = {10.1090/proc/16501},
       URL = {https://doi.org/10.1090/proc/16501},
}

@article {O1,
    AUTHOR = {Orponen, Tuomas},
     TITLE = {Additive properties of fractal sets on the parabola},
   JOURNAL = {Ann. Fenn. Math.},
  FJOURNAL = {Annales Fennici Mathematici},
    VOLUME = {48},
      YEAR = {2023},
    NUMBER = {1},
     PAGES = {113--139},
      ISSN = {2737-0690,2737-114X},
   MRCLASS = {28A80 (11B30 31E05 42B10)},
  MRNUMBER = {4528124},
       DOI = {10.54330/afm.125826},
       URL = {https://doi.org/10.54330/afm.125826},
}

@article {O2,
    AUTHOR = {Orponen, Tuomas and Puliatti, Carmelo and Py\"or\"al\"a,
              Aleksi},
     TITLE = {On {F}ourier transforms of fractal measures on the parabola},
   JOURNAL = {Trans. Amer. Math. Soc.},
  FJOURNAL = {Transactions of the American Mathematical Society},
    VOLUME = {378},
      YEAR = {2025},
    NUMBER = {10},
     PAGES = {7429--7450},
      ISSN = {0002-9947,1088-6850},
   MRCLASS = {28A80 (11B30 42B10)},
  MRNUMBER = {4956377},
       DOI = {10.1090/tran/9444},
       URL = {https://doi.org/10.1090/tran/9444},
}

@article {soly14/11,
    AUTHOR = {Solymosi, J\'ozsef},
     TITLE = {On the number of sums and products},
   JOURNAL = {Bull. London Math. Soc.},
  FJOURNAL = {The Bulletin of the London Mathematical Society},
    VOLUME = {37},
      YEAR = {2005},
    NUMBER = {4},
     PAGES = {491--494},
      ISSN = {0024-6093,1469-2120},
   MRCLASS = {11B75},
  MRNUMBER = {2143727},
MRREVIEWER = {Mei\ Chu\ Chang},
       DOI = {10.1112/S0024609305004261},
       URL = {https://doi.org/10.1112/S0024609305004261},
}

@article {orreg,
title={Sum-product phenomena for Ahlfors-regular sets},
author={O'Regan, William},
journal={arXiv preprint arXiv:2501.02131},
year={2025}
}

@article {sumprodent,
title={Frostman random variables, entropy inequalities, and applications},
author={Iosevich, Alex and Pham, Thang and Dac Quan, Nguyen  and Senger, Steven and Xue, Boqing},
journal={arXiv preprint arXiv:2507.15196},
year={2025}
}

@article {BoGu,
    AUTHOR = {Bourgain, Jean and Guth, Larry},
     TITLE = {Bounds on oscillatory integral operators based on multilinear
              estimates},
   JOURNAL = {Geom. Funct. Anal.},
  FJOURNAL = {Geometric and Functional Analysis},
    VOLUME = {21},
      YEAR = {2011},
    NUMBER = {6},
     PAGES = {1239--1295},
      ISSN = {1016-443X},
   MRCLASS = {42B20},
  MRNUMBER = {2860188},
MRREVIEWER = {Andrei K. Lerner},
       DOI = {10.1007/s00039-011-0140-9},
       URL = {https://doi.org/10.1007/s00039-011-0140-9},
}

@article {elekesruzsafewsums,
    AUTHOR = {Elekes, Gy. and Ruzsa, I. Z.},
     TITLE = {Few sums, many products},
   JOURNAL = {Studia Sci. Math. Hungar.},
  FJOURNAL = {Studia Scientiarum Mathematicarum Hungarica. Combinatorics,
              Geometry and Topology (CoGeTo)},
    VOLUME = {40},
      YEAR = {2003},
    NUMBER = {3},
     PAGES = {301--308},
      ISSN = {0081-6906,1588-2896},
   MRCLASS = {11B34 (05B10 11B75)},
  MRNUMBER = {2036961},
MRREVIEWER = {Georges\ Grekos},
       DOI = {10.1556/SScMath.40.2003.3.4},
       URL = {https://doi.org/10.1556/SScMath.40.2003.3.4},
}

@article {SolSz,
    AUTHOR = {Solymosi, J\'{o}zsef and Szab\'{o}, Endre},
     TITLE = {Classification of maps sending lines into translates of a
              curve},
   JOURNAL = {Linear Algebra Appl.},
  FJOURNAL = {Linear Algebra and its Applications},
    VOLUME = {668},
      YEAR = {2023},
     PAGES = {161--172},
      ISSN = {0024-3795},
   MRCLASS = {52C30 (17B81 52C10)},
  MRNUMBER = {4568803},
MRREVIEWER = {Chudamani Pranesachar Anil Kumar},
       DOI = {10.1016/j.laa.2023.03.010},
       URL = {https://doi.org/10.1016/j.laa.2023.03.010},
}

@article {Yi,
    AUTHOR = {Yi, Guangzeng},
     TITLE = {On bounded energy of convolution of fractal measures},
   JOURNAL = {Ann. Fenn. Math.},
  FJOURNAL = {Annales Fennici Mathematici},
    VOLUME = {50},
      YEAR = {2025},
    NUMBER = {2},
     PAGES = {437--457},
      ISSN = {2737-0690,2737-114X},
   MRCLASS = {42B10 (28A80 31B15)},
  MRNUMBER = {4941229},
       DOI = {10.54330/afm.163545},
       URL = {https://doi.org/10.54330/afm.163545},
}
\end{document}